\documentclass{amsart}

\usepackage{amsmath, amssymb}
\input xy
\xyoption{all}
\begin{document}

\newtheorem{innercustomthm}{Theorem}
\newenvironment{customthm}[1]
  {\renewcommand\theinnercustomthm{#1}\innercustomthm}
  {\endinnercustomthm}

% These will be typeset in italics
\newtheorem{theorem}{Theorem}[section]
\newtheorem{proposition}[theorem]{Proposition}
\newtheorem{lemma}[theorem]{Lemma}
\newtheorem{corollary}[theorem]{Corollary}
\newtheorem{fact}[theorem]{Fact}

% These will be typeset in Roman
\theoremstyle{definition}
\newtheorem{definition}[theorem]{Definition}
\newtheorem{conjecture}[theorem]{Conjecture}
\newtheorem{notation}[theorem]{Notation}
\newtheorem{claim}[theorem]{Claim}

\theoremstyle{remark}
\newtheorem{remark}[theorem]{Remark}
\newtheorem{example}[theorem]{Example}
\newtheorem{question}[theorem]{Question}

\numberwithin{equation}{section}

\def\diff{\operatorname{d}}
\def\tp{\operatorname{tp}}
\def\dme{\operatorname{DME}}
\def\gln{\operatorname{GL}_n}
\def\matn{\operatorname{Mat}_n}
\def\id{\operatorname{id}}
\def\alg{\operatorname{alg}}
\def\Frac{\operatorname{Frac}}
\def\Const{\operatorname{Const}}
\def\spec{\operatorname{Spec}}
\def\span{\operatorname{span}}
\def\exc{\operatorname{Exc}}
\def\Div{\operatorname{Div}}
\def\cl{\operatorname{cl}}
\def\mer{\operatorname{mer}}
\def\trdeg{\operatorname{trdeg}}
\def\ord{\operatorname{ord}}
\def\rank{\operatorname{rank}}
\def\loc{\operatorname{loc}}
\def\acl{\operatorname{acl}}
\def\dcl{\operatorname{dcl}}
\def\fdcf{\operatorname{FDCF}_{0,m}}
\def\dcf{\operatorname{DCF}_0}

\title[$D$-groups and the Dixmier-Moeglin Equivalence]{$D$-groups and the Dixmier-Moeglin Equivalence}

\author{Jason Bell}
\address{Department of Pure Mathematics, University of Waterloo, 200 University Avenue West, Waterloo, Ontario, Canada N2L 3G1}
\email{jpbell@uwaterloo.ca}

\author{Omar Le\'on S\'anchez}
\address{School of Mathematics, University of Manchester, Oxford Road, Manchester, United Kingdom M13 9PL}
\email{omar.sanchez@manchester.ac.uk}

\author{Rahim Moosa}
\address{Department of Pure Mathematics, University of Waterloo, 200 University Avenue West, Waterloo, Ontario, Canada N2L 3G1}
\email{rmoosa@uwaterloo.ca}
\date{\today}

\thanks{{\em Acknowledgements}: J. Bell and R. Moosa were partially supported by their respective NSERC Discovery Grants.}

\begin{abstract}
A differential-algebraic geometric analogue of the Dixmier-Moeglin equivalence is articulated, and proven to hold for $D$-groups over the constants.
The model theory of differentially closed fields of characteristic zero, in particular the notion of analysability in the constants, plays a central role.
As an application it is shown that if $R$ is a commutative affine Hopf algebra over a field of characteristic zero, and $A$ is an Ore extension to which the Hopf algebra structure extends, then $A$ satisfies the classical Dixmier-Moeglin equivalence.
Along the way it is shown that all such $A$ are Hopf Ore extensions in the sense of~[Brown et al., ``Connected Hopf algebras and iterated
Ore extensions", {\em Journal of Pure and Applied Algebra}, 219(6), 2015].
\end{abstract}

\maketitle

\setcounter{tocdepth}{1}
\tableofcontents

\section{Introduction}

\noindent
This article is about an analogue of the Dixmier-Moeglin equivalence for differential-algebraic geometry.
(The immediate motivation is an application to the classical noncommutative Dixmier-Moeglin problem, which we will describe later in this introduction.)
The main objects of study here are $D$-varieties.
An introduction to this category is given in~$\S$\ref{dvar-section}, but let us at least recall here that a $D$-variety (over the constants) is an algebraic variety $V$ over a field $k$ of characteristic zero, equipped with a regular section to the tangent bundle $s:V\to TV$.
A $D$-subvariety is an algebraic subvariety $W$ for which the restriction $s{\upharpoonright}_W$ is a section to the tangent bundle of $W$.
There are natural notions of $D$-morphism and $D$-rational map.
For convenience, let us assume that $k$ is algebraically closed.
We are interested in the following properties of an irreducible $D$-subvariety $W\subseteq V$ over $k$:
\vfill\pagebreak
\begin{itemize}
\item{\em $\delta$-primitivity}.
There is a $k$-point of $W$ that is not contained in any proper $D$-subvariety of $W$ over $k$.
\item{\em $\delta$-local-closedness}.
There is a maximum proper $D$-subvariety of $W$ over $k$.
\item{\em $\delta$-rationality}.
There is no nonconstant rational map from $(W,s)$ to $(\mathbb A^1,0)$ over $k$, where here $0$ denotes the zero section to the tangent bundle of the affine line.
\end{itemize}
The question is, for which ambient $D$-varieties $(V,s)$ are these three properties equivalent for all $D$-subvarieties?
It is not hard to see,  and is spelled out in the proof of Corollary~\ref{deduce-diffDME} below, that in general
$$\text{ $\delta$-local-closedness }\implies\text{ $\delta$-primitivity }\implies\text{ $\delta$-rationality.}$$
In earlier work~\cite{pdme}, together with St\'ephane Launois, we used the model theory of the Manin kernel to produce (in any dimension $\geq 3$) a $D$-variety which is itself $\delta$-rational but not $\delta$-locally-closed.
Here we focus on positive results.
The main one, which appears as Corollary~\ref{dgroupDME} below, is the following:
\begin{customthm}{A}
Suppose $(G,s)$ is a $D$-group over the constants -- that is, $G$ is an algebraic group and $s:G\to TG$ is a homomorphism of algebraic groups.
Then for any $D$-subvariety of $(G,s)$, $\delta$-rationality implies $\delta$-local-closedness.  In particular, for every $D$-subvariety of $(G,s)$, $\delta$-rationality, $\delta$-primitivity, and $\delta$-local-closedness are equivalent properties.
\end{customthm}

The proof of Theorem~A relies on the model theory of differentially closed fields.
In model-theoretic parlance, the point is that $\delta$-rationality of $(V,s)$ is equivalent to the generic type of the corresponding Kolchin closed set being weakly orthogonal to the constants, while $\delta$-local-closedness means that the type is isolated.
One context in which one can prove, using model-theoretic binding groups, for example, that weak orthogonality to the constants implies isolation, is when the type in question is analysable in the constants.
We give a geometric explanation of analysability in~$\S$\ref{ci-section} in terms of what we call {\em compound isotriviality} of $D$-varieties.
The reader can look there for a precise definition, but suffice it to say that a compound isotrivial $D$-variety is one that admits a finite sequence of fibrations where at each stage the fibres are isomorphic (possibly over a differential field extension of the base) to $D$-varieties where the section is the zero section.
We show that for compound isotrivial $D$-varieties $\delta$-rationality implies $\delta$-local-closedness
(Proposition~\ref{compoundisotrivial}).
Then we show, using known results about the structure of differential-algebraic groups, that every $D$-subvariety of a $D$-group over the constants is compound isotrivial (Proposition~\ref{ci}).
Theorem~A follows.

It turns out that for our intended application, namely Theorem~B2 appearing later in this introduction, we need Theorem~A to work for $D$-varieties that are slightly more general than $D$-groups.
Given an affine algebraic group $G$, we may as well assume that $G\subseteq\gln$,
so that a regular section to the tangent bundle is then of the form $s=(\id,\overline s)$ where $\overline s:G\to\matn$.
It is not hard to check from how the algebraic group structure is defined on the tangent bundle, that $s:G\to TG$ being a homomorphism is equivalent to the following identity:
$\overline s(gh)=\overline s(g)h+g\overline s(h)$,
where $g,h\in G$ are matrices and all addition and multiplication here is matrix addition and multiplication.
Now suppose we are given a homomorphism to the multiplicative group, $a:G\to\mathbb G_m$.
By an {\em $a$-twisted $D$-group} we mean a $D$-variety $(G,s)$ where $G$ is an affine algebraic group and $s=(\id,\overline s)$ satisfies the identity: $\overline s(gh)=\overline s(g)h+a(g)g\overline s(h)$.
So an $a$-twisted $D$-group is a $D$-group exactly when $a=1$.
We are able to show that $D$-subvarieties of $a$-twisted $D$-groups are also compound isotrivial.
This yields the following generalisation of Theorem~A: {\em For any $D$-subvariety of an $a$-twisted $D$-group over the constants, $\delta$-rationality, $\delta$-primitivity, and $\delta$-local-closedness are equivalent properties.}
The passage from $D$-groups to $a$-twisted $D$-groups turns out to be technically quite difficult, and is done in Section~\ref{atwist}.

It is worth pointing out that we have been intentionally ambigious about the field of definitions in the statements of Theorem~A and its $a$-twisted generalisation.
The reason for this is that the results actually hold true for $D$-subvarieties of $(G,s)$ that are defined over differential field extensions of the base field $k$.
To make this precise one has to give a more general definition of $D$-variety using prolongations rather than tangent bundles, and we have decided to delay this to the main body of the article.
While the final conclusion we are interested in is about $D$-subvarieties over~$k$, this possibility of passing to base extensions is an important part of the inductive arguments involved.

Now for the application to noncommutative algebra, to which Section~\ref{classical} is dedicated.
The classical Dixmier-Moeglin equivalence ($\dme$) is about prime ideals in a noetherian associative algebra over a field of characteristic zero; it asserts the equivalence between three properties of such prime ideals:  primitivity (a representation-theoretic property), local-closedness (a geometric property), and rationality (an algebraic property).
Precise definitions are given at the beginning of~$\S$\ref{classical}.
We are interested in the question of when the $\dme$ holds for skew polynomial rings $R[x;\delta]$ over finitely generated commutative integral differential $k$-algebras $(R,\delta)$.
Recall that $R[x;\delta]$ is the noncommutative polynomial ring in $x$ over $R$ where $xr=rx+\delta(r)$.
This question is not vacuous since examples of such skew polynomial rings failing the $\dme$ were given in~\cite{pdme}; indeed, these were the first counterexamples to the $\dme$ of finite Gelfand-Kirillov dimension.
The connection to $D$-varieties should be clear:  $R=k[V]$ for some irreducible algebraic variety $V$, and the $k$-linear derivation $\delta$ induces a regular section $s:V\to TV$.
So the study of such $(R,\delta)$ is precisely the same thing as the study of $D$-varieties.
We are able to prove (this is Proposition~\ref{ddme-sdme} below) that $R[x;\delta]$ will satisfy the $\dme$ if $\delta$-rationality implies $\delta$-locally-closedness for all $D$-subvarieties of the $D$-variety $(V,s)$ associated to $(R,\delta)$.
Theorem~A therefore answers our question in the special case of differential Hopf algebras. 

\begin{customthm}{B1}
If $(R,\delta)$ is a finitely generated commutative integral differential Hopf $k$-algebra then  $R[x;\delta]$ satisfies the $\dme$.
\end{customthm}
\noindent
Being a differential Hopf algebra means that $R$ has the structure of a Hopf algebra and that $\delta$ commutes with the coproduct -- this is equivalent to saying that $(R,\delta)$ comes from an affine $D$-group $(G,s)$.

More generally than skew polynomial rings, we consider Ore extensions: Suppose $R$ is a finitely generated commutative integral $k$-algebra, $\sigma$ is a $k$-algebra automorphism of $R$, and $\delta$ is a $k$-linear $\sigma$-derivation of $R$ -- meaning that $\delta(rs)=\sigma(r)\delta(s)+\delta(r)s$.
Recall that the Ore extension $R[x;\sigma,\delta]$ is the noncommutative polynomial ring in the variable $x$ over $R$ where $xr=\sigma(r)x+\delta(r)$.
So when $\sigma=\id$ we are in the skew polynomial case discussed above.
What about the $\dme$ for $R[x;\sigma,\delta]$?

\begin{customthm}{B2}
Suppose $R$ is a finitely generated commutative integral Hopf $k$-algebra.
If an Ore extension $R[x;\sigma,\delta]$ admits a Hopf algebra structure extending that on $R$, then $R[x;\sigma,\delta]$ satisfies the $\dme$.
\end{customthm}

\noindent
That Theorem~B1 is a special case of Theorem~B2 uses the (known) fact that one can always extend the Hopf structure on a differential algebra $R$ to the skew polynomial ring extension $R[x;\delta]$, namely by the coproduct induced by $\Delta(x)=x\otimes 1+1\otimes x$.
Theorem~B2 appears as Theorem~\ref{hodme} below.
Its proof goes via a reduction to the case when $\sigma=\id$ and then an application of the stronger $a$-twisted version of Theorem~A discussed above.
Both of these steps use the work of Brown et al.~\cite{bozz} on Hopf Ore extensions.
One obstacle is that while their results hold for much more general $R$ than we are considering, they are conditional on the coproduct of the variable $x$ in the Ore extension taking the special form
$$\Delta(x)=a\otimes x+x\otimes b+v(x\otimes x)+w$$
where $a,b\in R$ and $v,w\in R\otimes_kR$.
This is part of their definition of a Hopf Ore extension, though they speculate about its necessity.
We prove that when $k$ is algebraically closed, after a linear change of variable, $\Delta(x)$ always has the above form.
This is Theorem~\ref{deltax} below, and may be of independent interest:

\begin{customthm}{C}
Suppose $k$ is algebraically closed and $R$ is a finitely generated commutative integral Hopf $k$-algebra.
If an Ore extension $R[x;\sigma,\delta]$ admits a Hopf algebra structure extending that of $R$ then, after a linear change of the variable $x$, 
$$\Delta(x)=a\otimes x+x\otimes b+w$$
for some $a,b\in R$, each of which is either $0$ or group-like, and some $w\in R\otimes_k R$.
In particular, $R[x;\sigma,\delta]$ is a Hopf Ore extension of $R$.
\end{customthm}

It has been conjectured~\cite{BellLeung} that all finitely generated complex noetherian Hopf algebras of finite Gelfand-Kirillov dimension satisfy the $\dme$.
Theorem~B2 verifies a special case.
To make more significant progress on this conjecture one would like to pass from Hopf Ore extensions to iterated Hopf Ore extensions.
As of now, this appears to be beyond the scope of the techniques used here.

\medskip

Throughout this paper, by an {\em affine} $k$-algebra we mean a finitely generated commutative $k$-algebra that is an integral domain.

\medskip
\noindent
{\em Acknowledgements.}
We are grateful to an anonymous referee for a very thorough reading which lead to the discovery of an error in an initial version of this paper.

\bigskip
\section{The $\delta$-$\dme$ for $D$-groups over the constants}
\label{diffDME-section}

\noindent
In this chapter we prove Theorem~A of the introduction.
After some preliminaries, we articulate in~$\S$\ref{geometricddme} the differential-algebraic geometric analogue of the $\dme$ suggested in the introduction, and call it the $\delta$-$\dme$.
A sufficient condition for this to hold in terms of the model-theoretic notion of analysability to the constants is given in~$\S$\ref{ci-section}, and then applied to show that $D$-groups over the constants satisfy the $\delta$-$\dme$ in~$\S$\ref{dgroups-section}.
In a final section we reformulate $\delta$-$\dme$ algebraically, as a statement about commutative differential Hopf algebras, thereby preparing the stage for the application to the classical $\dme$ in chapter~\ref{classical}.

\medskip
\subsection{Preliminaries on  $D$-varieties}
\label{dvar-section}
\noindent
Suppose $k$ is a field of characteristic zero equipped with a derivation $\delta$.
In this section we review the notion of a $D$-variety over $k$.
Several more detailed expositions can be found in the literature, for instance Buium~\cite{Buium} who introduced the notion, and also~\cite[$\S$2]{kowalskipillay}.

%Several more detailed expositions can be found in the literature, at least in the case of $m=1$, see for example~\cite[$\S$2]{kowalskipillay}.

%It is convenient for us to work in a large (meaning sufficiently saturated) ambient existentially closed $\Delta$-field $(K,\Delta)$ extending $k$.
%In particular, $K$ is algebraically closed, and we identify all algebraic varieties with their $K$-points.

We first need to recall what prolongations are.
If $V\subseteq\mathbb A^n$ is an affine algebraic variety over $k$, then by the {\em $\delta$-prolongation} of $V$ is meant the algebraic variety $\tau V\subseteq \mathbb A^{2n}$ over $k$ whose defining equations are
\begin{eqnarray*}
P(X_1,\dots,X_n)&=&0\\
P^\delta(X_1,\dots,X_n)+\sum_{i=1}^n\frac{\partial P}{\partial X_i}(X_1,\dots,X_n)\cdot X_i&=&0
\end{eqnarray*}
for each $P\in I(V)\subset k[X_1,\dots,X_n]$.
Here $P^\delta$ denotes the polynomial obtained by applying $\delta$ to all the coefficients of $P$.
The projection onto the first $n$ coordinates gives us a surjective morphism $\pi:\tau V\to V$.

Note that if $K$ is any $\delta$-field extension of $k$, and $a\in V(K)$, then
$$\nabla(a):=(a,\delta a)\in\tau V(K).$$
If $V$ is defined over the constant field of $(k,\delta)$ then $\tau V$ is nothing other than the tangent bundle $TV$.
In general, $\tau V$ will be a torsor for the tangent bundle; for each $a\in V$ the fibre $\tau_a V$ is an affine translate of  the tangent space $T_aV$.
In particular, if $V$ is smooth and irreducible then so is $\tau V$.

Taking prolongations is a functor which acts on morphisms $f:V\to W$ by acting on their graphs.
It preserves the following properties of a morphism: \'etale-ness, being a closed embedding, and being smooth.
The functor $\tau$ acts naturally on rational maps also; this is because for $U$ a Zariski open subset of an irreducible variety $V$, $\tau V{\upharpoonright}_U=\tau(U)$ is Zariski open in $\tau(V)$.
Moreover, prolongations commute with base extension to $\delta$-field extensions.

We have restricted our attention here to the affine case merely for concreteness.
The prolongation construction extends to abstract varieties by patching over an affine cover in a natural and canonical way.

A {\em $D$-variety over $k$} is an algebraic variety $V$ over $k$ equipped with a regular section $s:V\to\tau V$ over $k$.
An example is when $V$ is defined over the constants and $s:V\to TV$ is the zero section.

If $V$ is affine then a $D$-variety structure on $V$ is nothing other than an extension of $\delta$ to the co-ordinate ring $k[V]$.
Indeed,  if $s:V\to \tau V$ is given by $s(X)=\big(X,s_1(X),\dots,s_n(X)\big)$ in variables $X=(X_1,\dots,X_n)$, then we can extend $\delta$ to $k[X]$ by $X_j\mapsto s_j(X)$, and this will induce a derivation on $k[V]$.
Conversely, given an extension of $\delta$ to $k[V]$, and choosing $s_j(X)$ to be such that $\delta\big(X_j+I(V)\big)=s_j(X)+I(V)$, we get that $s:=\big(\id,s_1,\dots,s_n\big):V\to\tau V$ is a regular section.

A {\em $D$-subvariety} of $(V,s)$ is a closed algebraic subvariety $W\subseteq V$, over a possibly larger $\delta$-field $K$, such that $s(W)\subseteq\tau W$.
In principle one should talk about the base extension of $V$ to $K$ before talking about subvarieties {\em over $K$}, but as prolongations commute with base extension, and following standard model-theoretic practices, we allow $D$-subvarieties to be defined over arbitrary $\delta$-field extensions unless explicitly stated otherwise.

A $D$-variety $(V,s)$ over $k$ is said to be {\em $k$-irreducible} if $V$ is $k$-irreducible as an algebraic variety.
In this case $s$ induces on $k(V)$ the structure of a $\delta$-field extending~$k$.
A $D$-variety $(V,s)$ is called {\em irreducible} if $V$ is absolutely irreducible.
In general, every irreducible component of $V$ is a $D$-subvariety over $k^{\alg}$ and these are called the {\em irreducible components} of $(V,s)$.

A {\em morphism} of $D$-varieties $(V,s)\to (W,t)$ is a morphism of algebraic varieties $f:V\to W$ such that
$$\xymatrix{
\tau V\ar[r]^{\tau f}&\tau W\\
V\ar[u]^s\ar[r]^f&W\ar[u]_t
}$$
commutes.
It is not hard to verify that the pull-back of a $D$-variety, and the Zariski closure of the image of a $D$-variety, under a $D$-morphism, are again $D$-varieties.

In the same way, we can talk about {\em rational maps} between $D$-varieties.
A useful fact is that if $U$ is a nonempty Zariski open subset of $V$, then the prolongation of $U$ is the restriction of $\tau V$ to $U$, and so $(U,s{\upharpoonright}_U)$ is a $D$-variety in its own right.
So a rational map on $V$ is a $D$-rational map if it is a $D$-morphism when restricted to the Zariski open subset on which it is defined.

A {\em $D$-constant} rational function on a $D$-variety $(V,s)$ over $k$ is a rational map over $k$ from $(V,s)$ to $(\mathbb A^1,0)$ where $0$ denotes the zero section to the tangent bundle of the affine line.
In the case when $(V,s)$ is $k$-irreducible, they correspond precisely to the $\delta$-constants of $\big(k(V),\delta\big)$.

\medskip
\subsection{Differentially closed fields and the Kolchin topology}
Underlying our approach to the study of $D$-varieties is the model theory of {\em existentially closed} $\delta$-fields (of characteristic zero).
These are $\delta$-fields $K$ with the property that any finite sequence of $\delta$-polynomial equations and inequations over $K$ which have a solution in some $\delta$-field extension, already have a solution in $K$.
The class of existentially closed $\delta$-fields of characteristic zero is axiomatisable in first-order logic, and its theory is denoted by $\operatorname{DCF}_0$.
We will work in a fixed model of this theory, an existentially closed $\delta$-field~$K$.
In particular, $K$ is algebraically closed.
We let $K^\delta$ denote the field of constants of $K$; it is an algebraically closed field that is {\em pure} in the sense that the structure induced on it by $\dcf$ is simply that given by the language of rings.

Suppose $(V,s)$ is a $D$-variety over $K$.
Let $x\in V(K)$.
Note that $\{x\}$ is a $D$-subvariety if and only if $\nabla(x)=s(x)$ where recall that $\nabla:V(K)\to\tau V(K)$ is given by $x\mapsto(x,\delta x)$.
We call such points {\em $D$-points}, and denote the set of all $D$-points in $V(K)$ by $(V,s)^\sharp(K)$.
It is an example, the main example we will encounter, of a {\em Kolchin closed} subset of $V(K)$.
In general a Kolchin closed subset of the $K$-points of an algebraic variety is one that is defined Zariski-locally  by the vanishing of $\delta$-polynomials.
Note that when $(V,s)$ is defined over the constants and $s$ is the zero section, $(V,s)^\sharp(K)=V(K^\delta)$.

One of the main consequences of working in an existential closed $\delta$-field is that $(V,s)^\sharp(K)$ is Zariski-dense in $V(K)$.
In particular, for any  subvariety $W\subseteq V$, we have that $W$ is a $D$-subvariety if and only if $W\cap(V,s)^\sharp(K)$ is Zariski dense in $W(K)$.

Suppose $(V,s)$ is $k$-irreducible for some $\delta$-subfield $k$.
If we allow ourselves to pass to a larger existentially closed $\delta$-field, then we can always find a {\em $k$-generic $D$-point of $(V,s)$}, that is, a $D$-point $x\in(V,s)^\sharp(K)$ that is Zariski-generic over~$k$ in~$V$.
Note that such a point is also Kolchin-generic in $(V,s)^\sharp(K)$ over $k$ in the sense that it is not contained in any proper Kolchin closed subset defined over $k$.

In order to ensure the existence of generic $D$-points without having to pass to larger $\delta$-fields, it is convenient to assume that $K$ is already sufficiently large, namely {\em saturated}.
This means that if $k$ is a $\delta$-subfield of strictly smaller cardinality than $K$, and $\mathcal F$ is a collection of Kolchin constructible sets over $k$ every finite subcollection of which has a nonempty intersection, then $\mathcal F$ has a nonempty intersection.

\medskip
\subsection{An analogue of the $\dme$ for $D$-varieties}
\label{geometricddme}
We fix from now on a saturated existentially closed $\delta$-field $K$ of sufficiently great cardinality.
So $K$ serves as a universal domain for $\delta$-algebraic geometry (and hence, in particular, algebraic geometry).
We also fix a small $\delta$-subfield $k\subset K$ that will serve as the field of coefficients.

\begin{definition}[$\delta$-$\dme$ for $D$-varieties]
\label{deltadme-geometric}
Suppose $(V,s)$ is a $D$-variety over $k$.
We say that $(V,s)$ {\em satisfies the $\delta$-$\dme$ over $k$}, if for every $k$-irreducible $D$-subvariety $W\subseteq V$, the following are equivalent:
\begin{itemize}
\item[(i)]
$(W,s)$ is {\em $\delta$-primitive}:
there exists a point $p\in W(k^{\alg})$ that is not contained in any proper $D$-subvariety of $W$ over $k$.
\item[(ii)]
$(W,s)$ is {\em $\delta$-locally-closed}:
it has a maximum proper $D$-subvariety over $k$.
\item[(iii)]
$(W,s)$ is {\em $\delta$-rational}:
$k(W)^\delta\subseteq k^{\alg}$.
\end{itemize}
\end{definition}

\begin{remark}
As the model-theorist will notice, and as we will prove in the next section, $W$ being $\delta$-locally closed means that the Kolchin generic type $p$ of $(W,s)^\sharp(K)$ over $k$ is isolated.
The model-theoretic meaning of $\delta$-rationality is that $p$ is weakly orthogonal to the constants.
On the other hand, it is not clear how to express {\em a priori} the $\delta$-primitivity of $W$ as a model-theoretic property of $p$.
\end{remark}

Without additional assumptions on $k$ there is no hope for the $\delta$-$\dme$ being satisfied.
For example, there are positive-dimensional $\delta$-rational $D$-varieties over any $k$, but if $k$ is differentially closed then the only $\delta$-locally closed $D$-varieties over $k$ are zero-dimensional.
This is because every $D$-variety over a differentially closed field $k$ will have a Zariski dense set of $D$-points over $k$, and so a $D$-subvariety over $k$ containing all of them could not be proper.
We are interested, however, in the case when $k$ is very much not differentially closed; namely, when $\delta$ is trivial on $k$.

\begin{proposition}
\label{knownparts}
For any $k$-irreducible $D$-variety, $\delta$-local-closedness implies $\delta$-primitivity.
Moreover, if $k\subseteq K^\delta$ then $\delta$-primitivity implies $\delta$-rationality.
\end{proposition}

\begin{proof}
Let $(W,s)$ be a $k$-irreducible $D$-variety.

Suppose $(W,s)$ is $\delta$-locally-closed, and denote by~$A$ the maximum proper $D$-subvariety of $W$ over $k$.
Then $p\in W(k^{\alg})\setminus A(k^{\alg})$ witnesses $\delta$-primitivity. This proves the first assertion.

Now suppose that $k\subseteq K^\delta$, $(W,s)$ is $\delta$-primitive, and $p\in W(k^{\alg})$ is not contained in any proper $D$-subvariety over $k$.
Suppose $f\in k(W)$ is a $\delta$-constant.
We want to show that $f\in k^{\alg}$.
We view it as a rational map of $D$-varieties, $f:(W,s)\to (\mathbb A^1,0)$, and suppose for now that it is defined at $p$.
So $f(p)\in \mathbb A^1(k^{\alg})$.
Because of our additional assumption that $k\subseteq K^\delta$, and hence $k^{\alg}\subseteq K^\delta$, we have that $f(p)$ is a $D$-point of $(\mathbb A^1,0)$.
Now, let $\Lambda$ be the orbit of $f(p)$ under the action of the absolute Galois group of $k$.
Then $\Lambda$ is a finite $D$-subvariety of $(\mathbb A^1,0)$ over $k$.
Hence the Zariski closure of $f^{-1}(\Lambda)$ is a $D$-subvariety of $W$ over $k$ that contains $p$.
It follows that $f^{-1}(\Lambda)=W$.
So $f$ is $k^{\alg}$-valued on all of $W$.
We have shown that every element of $k(W)^\delta$ that is defined at $p$ is in $k^{\alg}$.

We have still to deal with the possibility that $f$ is not defined at $p$.
In that case, writing $f=\frac{a}{b}$ with $a,b\in k[W]$, we must have $b(p)=0$.
The fact that $\delta f=0$ implies by the quotient rule that $b\delta a=a\delta b$.
So either $\delta a=\delta b=0$ or $f=\frac{a}{b}=\frac{\delta a}{\delta b}$.
Since $a$ and $b$ are defined at $p$, if $\delta a=\delta b=0$ then $a,b\in k^{\alg}$ by the previous paragraph, and hence $f\in k^{\alg}$.
If, on the other hand, $f=\frac{\delta a}{\delta b}$, then we iterate the argument with $(\delta a, \delta b)$ in place of $(a,b)$.
What we get in the end is that either $f\in k^{\alg}$ or $f=\frac{\delta^\ell a}{\delta^\ell b}$ for all $\ell\geq 0$.
We claim the latter is impossible.
Indeed, it would imply that $\delta^\ell b(p)=0$ for all $\ell$, and so $p$ is contained in the $D$-subvariety $V(I)$ where $I$ is the $\delta$-ideal of $k[W]$ generated by $b$.
But the assumption on $p$ would then imply that $V(I)=W$, contradicting the fact that $b\neq 0$.
So $f\in k^{\alg}$, as desired.
\end{proof}

So the question becomes:

\begin{question}
\label{dme-question}
Under the assumption that $\delta$ is trivial on $k$, for which $D$-varieties does $\delta$-rationality imply $\delta$-local-closedness?
\end{question}

Question~\ref{dme-question} should be, we think, of general interest in differential-algebraic geometry.
In~\cite{pdme} it was pointed out that Manin kernels can be used to construct, in all Krull dimensions at least three, examples that were $\delta$-rational but not $\delta$-locally closed.
Let us point out that in dimension $\leq 2$ the answer is affirmative:

\begin{proposition}
\label{dmeleq2}
If $k\subseteq K^\delta$ then every $D$-variety over $k$ of dimension $\leq 2$ satisfies the $\delta$-$\dme$ over $k$.
\end{proposition}

\begin{proof}
Suppose $(V,s)$ is a $D$-variety over $k$ of dimension at most $2$.
By Proposition~\ref{knownparts} it suffices to show that if $W\subseteq V$ is a $k$-irreducible $\delta$-rational $D$-subvariety over $k$ then it has a maximum proper $D$-subvariety over $k$.
We may assume that $\dim W>0$.
Now, it is a known fact that $\delta$-rationality implies the existence of only finitely many $D$-subavrieties of codimension one over $k$.
Indeed, this is an unpublished theorem of Hrushovski~\cite[Proposition~2,3]{hrushovski-jouanolou}; see~\cite[Theorem~6.1]{pdme} and~\cite[Theorem~4.2]{freitagmoosa} for published generalisations.
So it remains to consider the $0$-dimensional $D$-subvarieties of $W$ over $k$.
But as $k\subseteq K^\delta$, the union of these is contained in the Zariski closure $X$ of $(W,s)^\sharp(K)\cap W(K^\delta)$.
Note that $s$ restricts to the zero section on $X$, and hence $X$ is a $D$-subvariety of $W$ over $k$ that must be proper by $\delta$-rationality of $(W,s)$.
So the union of $X$ and the finitely many codimension one $D$-subvarieties of $W$ form the maximal proper $D$-subvariety over $k$.
\end{proof}

We will give a sufficient condition for $\delta$-rationality to imply $\delta$-local-closedness, and hence for the $\delta$-$\dme$, having to do with analysability to the constants in the model theory of differentially closed fields.
We will then use this condition  to prove the $\delta$-$\dme$ for $D$-groups over the constants.

\medskip
\subsection{Maximum $D$-subvarieties}
Here we look closer at which $D$-varieties over~$k$ have a proper $D$-subvariety over~$k$ that contains all other proper $D$-subvarieties over~$k$.
This is something that never happens in the pure algebraic geometry setting: every variety over $k$ has a Zariski dense set of $k^{\alg}$-points, and each $k^{\alg}$-point is contained in a finite subvariety defined over $k$.
So a $k$-irreducible variety cannot have a maximum proper subvariety over~$k$.
In the enriched context of $D$-varieties there will be many $D$-points over a differential closure of $k$, say $\widetilde k$, but a $\widetilde k$-point need not live in a proper $D$-subvariety defined over $k$.
So $D$-points are not an obstacle to the existence of a maximum proper $D$-subvariety.
In fact, as the following lemma  points out, the existence of a maximum proper $D$-subvariety is a natural property to consider from both the Kolchin topological and model-theoretic points of view.

We continue to work in our sufficiently saturated differentially closed field $(K,\delta)$, and fix a $\delta$-subfield $k$ of coefficients.

\begin{lemma}
\label{equivisolated}
Suppose $(V,s)$ is a $k$-irreducible $D$-variety.
The following are equivalent.
\begin{itemize}
\item[(i)]
$(V,s)$ is $\delta$-locally closed.
\item[(ii)]
$(V,s)$ has finitely many maximal proper $k$-irreducible $D$-subvarieties.
\item[(iii)]
$\Lambda:=(V,s)^\sharp(K)\setminus\bigcup\{(W,s)^\sharp(K):W\subsetneq V\text{ $D$-subvariety over $k$}\}$ is Kolchin constructible.
\item[(iv)]
The Kolchin generic type of $(V,s)^\sharp(K)$ over $k$ is an isolated type.
\end{itemize}
\end{lemma}

\begin{proof}
(i)$\implies$(ii).
Let $W$ be the maximum proper $D$-subvariety over $k$.
The $k$-irreducible components of $W$ are $D$-subvarieties of $V$, see for example~\cite[Theorem~2.1]{kaplansky}.
Every proper $k$-irreducible $D$-subvariety of $V$ is contained in one of these components.
So the maximal proper $k$-irreducible $D$-subvarieties of $V$ are precisely the $k$-irreducible components of $W$.

(ii)$\implies$(iii).
Let $W_1,\dots, W_\ell$ be the maximal proper $k$-irreducible $D$-subvarieties of $V$.
Then
$$\bigcup\{(W,s)^\sharp(K):W\subsetneq V\text{ $D$-subvariety over $k$}\}=(W_1,s)^\sharp(K)\cup\dots\cup(W_\ell,s)^\sharp(K).$$

(iii)$\implies$(iv).
The Kolchin generic type of $(V,s)^\sharp(K)$ over $k$ is the complete type $p(X)$ in $\dcf$ axiomatised by the formulas saying that ``$X\in (V,s)^\sharp(K)$", and, for each proper Kolchin closed subset $A$ of $(V,s)^\sharp(K)$ over $k$, the formula ``$X\not\in A$".
Note that as $(V,s)^\sharp(K)$ is defined by $\nabla(X)=s(X)$, the occurrences of each $\delta X$ in the defining equations of $A$ can be replaced by polynomials, so that $A=A^{\operatorname{Zar}}\cap(V,s)^\sharp(K)$, where $A^{\operatorname{Zar}}$ denotes the Zariski closure of $A$ in $V$.
When $A$ is over $k$, we have that $A^{\operatorname{Zar}}$ is a $D$-subvariety of $V$ over $k$.
It follows that the set of realisations of $p$ is precisely $\Lambda$, so that $\Lambda$ being Kolchin constructible implies that $p$ is axiomatised by a single formula, that is, it is isolated.

(iv)$\implies$(i).
Let $\Lambda$ be as in statement~(iii).
As we have seen, this is the set of realisations of the Kolchin generic type of $(V,s)^\sharp(K)$ over $k$.
The latter being isolated implies, by quantifier elimination, that $\Lambda$ is Kolchin constructible.
By saturation, this in turn implies that
$A:=\bigcup\{(W,s)^\sharp(K):W\subsetneq V\text{ $D$-subvariety over $k$}\}$ is a finite union, and hence is itself a proper Kolchin closed subset over $k$.
Then $A^{\operatorname{Zar}}$ is the maximum proper $D$-subvariety over $k$.
\end{proof}

The following lemma will be useful in showing that certain $D$-varieties satisfy the equivalent conditions of Lemma~\ref{equivisolated}.

\begin{lemma}
\label{basefibre}
Suppose $f:(V,s)\to (W,t)$ is a dominant $D$-rational map of $k$-irreducible $D$-varieties over $k$.
The following are equivalent
\begin{itemize}
\item[(i)]
$(V,s)$ has a maximum proper $D$-subvariety over $k$.
\item[(ii)]
$(W,t)$ has a maximum proper $D$-subvariety over $k$, and for some (equivalently every) $k$-generic $D$-point $\eta$ of $W$, the fibre $V_\eta:=f^{-1}(\eta)^{\operatorname{Zar}}$ has a maximum proper $D$-subvariety over $k(\eta)$.
\end{itemize}
\end{lemma}

\begin{proof}
We show how this follows easily from basic properties of isolated types, leaving it to the reader to make the straightforward, but rather unwieldy, translation it into an algebro-geometric argument if desired.

(i)$\implies$(ii).
Suppose $a\in (V,s)^\sharp(K)$ is a $k$-generic $D$-point.
Since $f$ is a dominant $D$-rational map, $f(a)\in(W,T)^\sharp(K)$ is also $k$-generic.
By characterisation~(iv) of the previous lemma, $\tp(a/k)$ is isolated.
As $f(a)$ is in the definable closure of $a$ over $k$, it follows that $\tp(f(a)/k)$ is isolated.
Hence, $(W,t)$ has a maximum proper $D$-subvariety over $k$.

Now fix $\eta\in (W,t)^\sharp(K)$ a $k$-generic $D$-point,
and let $a$ be a $k(\eta)$-generic $D$-point of the fibre $V_\eta$.
Then $a$ is $k$-generic in $(V,t)$, and hence $\tp(a/k)$ is isolated.
It follows that the extension $\tp(a/k(\eta))$ is also isolated.
So $V_\eta$ has a maximum proper $D$-subvariety over $k(\eta)$.

(ii)$\implies$(i).
Fix $\eta\in (W,t)^\sharp(K)$ a $k$-generic $D$-point such that $V_\eta$ has a maximum proper $D$-subvariety over $k(\eta)$.
Let $a$ be a $k(\eta)$-generic $D$-point of $V_\eta$.
So $\tp(a/k(\eta))$ and $\tp(\eta/k)$ are both isolated, implying that $\tp(a/k)$ is isolated.
Since $a\in(V,s)^\sharp(K)$ is $k$-generic, condition (i) follows.
\end{proof}

\medskip
\subsection{Compound isotriviality}
\label{ci-section}
Our sufficient condition for $\delta$-rationality to imply $\delta$-local-closedness will come from looking at isotrivial $D$-varieties.

\begin{definition}
An irreducible $D$-variety $(V,s)$ over $k$ is said to be {\em isotrivial} if there is some $\delta$-field extension $F\supseteq k$ such that $(V,s)$ is $D$-birationally equivalent over $F$ to a $D$-variety of the form $(W,0)$ where $W$ is defined over the constants $F^\delta$ and $0$ is the zero section.
We will say that a possibly reducible $D$-variety is isotrivial if every irreducible component is.
\end{definition}

This definition comes from model theory, it is a geometric translation of the statement that the Kolchin generic type of $(V,s)^\sharp(K)$ over $k$ is $K^\delta$-internal.
Note that there is some tension, but no inconsistency, between isotriviality and $\delta$-rationality; for example, $(W,0)$ is far from being $\delta$-rational, instead of there being no new $\delta$-constants in the rational function field we have that $\delta$ is trivial on all of $k(W)$.
The reasons these notions are not inconsistent is that the isotrivial $(V,s)$ is only of the form $(W,0)$ after base change -- that is, over additional parameters -- and that makes all the difference.

\begin{fact}
\label{isotrivial}
A $k$-irreducible $D$-variety that is at once both $\delta$-rational and isotrivial must be $\delta$-locally closed.
\end{fact}

\begin{proof}
Suppose $(V,s)$ is a $k$-irreducible isotrivial $D$-variety with $k(V)^\delta\subseteq k^{\alg}$.
We want to show that $V$ has a maximum proper $D$-subvariety over $k$.
The proof we give makes essential use of model theory.
We will show how the statement translates to the well-known fact that a type internal to the constants but weakly orthogonal to the constants is isolated.

Let $p$ be the Kolchin generic type of $(V,s)^\sharp(K)$ over $k$.
By Lemma~\ref{equivisolated}, it suffices to show that $p$ is isolated.
That in turn reduces to showing that every extension of $p$ to $k^{\alg}$ is isolated.
Fix $q$ an extension of $p$ to~$k^{\alg}$.
So $q$ is the Kolchin generic type of $(\widehat V,s)^\sharp(K)$ over $k^{\alg}$, for some irreducible component $\widehat V$ of $V$.
Isotriviality of $(V,s)$ implies isotriviality of $(\widehat V,s)$,  and this means that $q$ is internal to the constants $K^\delta$, see for example~\cite[Fact~2.6]{kowalskipillay}.
By stability, this implies that the binding group $\mathcal G=\operatorname{Aut}(q/k^{\alg}(K^\delta))$ is type-definable over $k^{\alg}$, see for instance \cite[Appendix B]{Hrushovski}.
In fact, $\mathcal G$ is definable: $\mathcal G$ lives in the constants and by $\omega$-stability of the induced structure on the constants, every type-definable group in $K^\delta$ is a definable group.
So we have a definable group acting definably on the set of realisations of $q$.

On the other hand, for all $a\models q$ we have that 
$$k^{\alg}(a)^\delta\subseteq (k(a)^{\alg})^\delta= (k(a)^\delta)^{\alg}\subseteq k^{\alg},$$
where the last containment uses our assumption on $k(V)=k(a)$.
This shows that $q$ is weakly orthogonal to $K^\delta$. 
So the action of $\mathcal G$ on the set of realisations of $q$ is transitive.
As $\mathcal G$ is definable, the set of realisations of $q$ must be definable -- that is, $q$ is isolated.
\end{proof}

Using Lemma~\ref{basefibre} we can extend Fact~\ref{isotrivial} to the case of $D$-varieties that are built up by a finite sequence of fibrations by isotrivial $D$-subvarieties.

\begin{definition}
An irreducible $D$-variety $(V,s)$ over $k$ is said to be {\em compound isotrivial} if there exists a sequence of irreducible $D$-varieties $(V_i,s_i)$ over $k$, for $i=0,\dots,\ell$, with dominant $D$-rational maps over $k$
$$
\xymatrix{
V=V_0\ar[r]^{ \ \ \ f_0}&V_1\ar[r]^{f_1}&\cdots\ar[r]&V_{\ell-1}\ar[r]^{f_{\ell-1} \ \ \ }&V_\ell=0}
$$
where $0$ denotes an irreducible zero-dimensional $D$-variety,
and such that the generic fibres of each $f_i$ are isotrivial.
That is, for each $i=0,\dots,\ell-1$, if $\eta$ is a $k$-generic $D$-point in $V_{i+1}$, then $f_i^{-1}(\eta)^{\operatorname{Zar}}$, which is a $k(\eta)$-irreducible $D$-subvariety of $(V_i,s_i)$, is isotrivial.
We say $(V,s)$ is compound isotrivial in {\em $\ell$ steps}.
\end{definition}

While isotriviality is equivalent to the Kolchin generic type being internal to the constants, compound isotriviality corresponds to that type being {\em analysable} in the constants.
As this is a less familiar notion, even among model theorists, we spell out the equivalence here.

\begin{lemma}
\label{cianal}
Suppose $(V,s)$ is an irreducible $D$-variety over $k$, and $a\in (V,s)^\sharp(K)$ is a $k$-generic $D$-point.
Then $(V,s)$ is compound isotrivial if and only if the type of $a$ over $k$ in $\dcf$ is analysable in $K^\delta$.
\end{lemma}
 
 \begin{proof}
 Analysability in $K^\delta$ means that there are tuples $a=a_0,a_1,\dots,a_\ell$, such that
 \begin{itemize}
 \item[(i)]
 $a_i$ is in the $\delta$-field generated by $a_{i-1}$ over $k$, for $i=1,\dots,\ell-1$, $a_\ell\in k$, and
 \item[(ii)]
 the type of $a_i$ over the algebraic closure of the $\delta$-field generated by $k(a_{i+1})$ is internal to $K^\delta$.
 \end{itemize}
 If $(V,s)$ is compound isotrivial one simply takes $a_{i}=f_{i-1}(a_{i-1})$ for $i=1,\dots,\ell$.
 Condition~(i) is clear -- in fact with ``$\delta$-field generated by" replaced by ``field generated by" -- and condition~(ii) follows from the fact that $a_i$ will be a $k(a_{i+1})^{\alg}$-generic $D$-point of one of the irreducible components of  $f_i^{-1}(a_{i+1})^{\operatorname{Zar}}$, all of which are isotrivial.
For the converse, given $a=a_0,a_1,\dots,a_\ell$ satisfying condition~(i) and~(ii), one first replaces $a_i$ with $(a_i,\delta(a_i),\dots,\delta^n(a_i))$ for some sufficiently large $n$ so that $a_i$ is a $k$-generic $D$-point of an irreducible $D$-variety $(V_i,s_i)$ over $k$.
This sequence of $D$-varieties will witness the compound isotriviality, using the fact that the irreducible components of $f_i^{-1}(a_{i+1})^{\operatorname{Zar}}$ are all conjugate over $k(a_{i+1})$ and hence the isotriviality of one implies the isotriviality of them all.
 \end{proof}
 
 \begin{remark}[Stability under base change]
The definition of compound isotriviality seems to be sensitive to parameters; the $D$-varieties $V_i$ and the $D$-rational maps $f_i$ need also be defined over~$k$.
In fact the notion is stable under base change: if an irreducible $D$-variety $(V,s)$ over~$k$ is compound isotrivial when viewed as a $D$-variety over some $\delta$-field extension $F\supseteq k$ then it was already compound isotrivial over $k$.
A model-theoretic re-statement of this is the well-known fact that a stationary type with a nonforking extension that is analysable in the constants is already analysable in the constants.
We leave it to the reader to formulate a geometric argument.
\end{remark}

Note also that (compound) isotriviality is preserved by $D$-birational maps.

\begin{proposition}
\label{compoundisotrivial}
For an irreducible compound isotrivial $D$-variety over $k$, \newline $\delta$-rationality implies $\delta$-local-closedness.
\end{proposition}

\begin{proof}
Suppose $(V,s)$ is an irreducible compound isotrivial $D$-variety over $k$ with $k(V)^\delta\subseteq k^{\alg}$.
We need to show that $V$ has a maximum proper $D$-subvariety over~$k$.
We proceed by induction on the number of steps witnessing the compound isotriviality.
The case $\ell=0$ is vacuous.
Suppose we have a compound isotrivial $(V,s)$ witnessed by
$$
\xymatrix{
V=V_0\ar[r]^{ \ \ \ f_0}&V_1\ar[r]^{f_1}&\cdots\ar[r]&V_{\ell-1}\ar[r]^{f_{\ell-1} \ \ \ }&V_\ell=0}
$$
with $\ell\geq1$.
Then $V_1$ is compound isotrivial in $\ell-1$ steps, and as $k(V_1)$ is a $\delta$-subfield of $k(V)$ by dominance of $f_0$, the induction hypothesis applies to give us a maximum proper $D$-subvariety of $V_1$ over $k$.

On the other hand, the generic fibre $V_\eta:=f_0^{-1}(\eta)^{\operatorname{Zar}}$ is an isotrivial $k(\eta)$-irreducible $D$-subvariety of $V$; where $\eta$ is a $k$-generic $D$-point of $V_1$.
Moreover, as $k(\eta)(V_\eta)=k(V)$, $V_\eta$ is $\delta$-rational and therefore Fact~\ref{isotrivial} applies to $V_\eta$ and we obtain a maximum proper $D$-subvariety over $k(\eta)$.
Now Lemma~\ref{basefibre} implies that $V$ has a maximum proper $D$-subvariety over $k$.
\end{proof}

\begin{corollary}
\label{deduce-diffDME}
Suppose $k\subseteq K^\delta$ and $(V,s)$ is a $D$-variety over $k$ with the property that every irreducible $D$-subvariety of $V$ over~$k^{\alg}$ is compound isotrivial.
Then $(V,s)$ satisfies $\delta$-$\dme$.
\end{corollary}

\begin{proof}
By Proposition~\ref{knownparts}, it suffices to show that every $\delta$-rational $k$-irreducible $D$-subvariety $(W,s)$ is $\delta$-locally closed.
Note that if $k=k^{\alg}$ then $(W,s)$ is absolutely irreducible, and compound isotrivial by assumption, so that $\delta$-local-closedness follows by Proposition~\ref{compoundisotrivial}.
In general, let $(W_0,s)$ be an absolutely irreducible component of $(W,s)$.
It is over $k^{\alg}$.
The $\delta$-rationality of $(W,s)$ over $k$ implies the $\delta$-rationality of $(W_0,s)$ over $k^{\alg}$ -- see for example the last paragraph of the proof of Fact~\ref{isotrivial}.
By assumption $(W_0,s)$ is compound isotrivial, and so by Proposition~\ref{compoundisotrivial} it is $\delta$-locally closed over $k^{\alg}$.
We have shown that every irreducible component of $(W,s)$ is $\delta$-locally closed over $k^{\alg}$, and it is not hard to see, by taking the union of the maximum proper $D$-subvarieties of these components, for example, that this implies that $(W,s)$ is $\delta$-locally closed, as desired.
\end{proof}

\medskip
\subsection{$D$-groups over the constants}
\label{dgroups-section}
A {\em $D$-group} is a $D$-variety $(G,s)$ over $k$ whose underlying variety $G$ is an algebraic group, and such that the section $s:G\to \tau G$ is a morphism of algebraic groups.
(Note that there is a unique algebraic group structure on $\tau(G)$ which makes the embedding $\nabla:G(K)\to\tau G(K)$ a homomorphism.)
The notions of {\em $D$-subgroup} and {\em homomorphism of $D$-groups} are the natural ones, with the caveat that, unless stated otherwise, parameters may come from a larger $\delta$-field.
The quotient of a $D$-group by a normal $D$-subgroup admits a natural $D$-group structure.
The terms {\em connected} and {\em connected component of identity} when applied to $D$-groups refer just to the underlying algebraic group, though note that the connected component of identity of a $D$-group over $k$ is a $D$-subgroup.

In the context of $D$-groups isotriviality is better behaved.
A connected $D$-group $(G,s)$ is isotrivial if and only if it is isomorphic as a $D$-group to one of the form $(H,0)$ where $H$ is an algebraic group over the constants and $0$ is the zero section.
So one remains in the category of $D$-groups, and $D$-birational equivalence is replaced by $D$-isomorphism.
See the discussion around Fact~2.6 of~\cite{kowalskipillay} for a proof of this.
In particular, every $D$-subvariety of an isotrivial $D$-group is itself isotrivial.
Quotients of isotrivial $D$-groups are also isotrivial.
Moreover, by~\cite[Corollary~3.10]{pillay06}, if a $D$-group $(G,s)$ has a finite normal $D$-subgroup $H$ such that $G/H$ is isotrivial, then $(G,s)$ must have been isotrivial to start with.
We also note that, as (compound) isotriviality is preserved under $D$-birational maps, when working inside a $D$-group (compound) isotriviality is preserved under translation by $D$-points of $G$ (as these translations will in fact be $D$-automorphisms of $G$).

The following fact is mostly a matter of putting together various results in the literature on $D$-groups.
As we will see, it will imply that every $D$-subvariety of a $D$-group over the constants is compound isotrivial in at most $3$ steps. At this point it is worth noting that the set of $D$-points of a $D$-group is a subgroup definable in $\dcf$ of finite Morley rank. Moreover, the $\sharp$ functor is an equivalence between the categories of $D$-groups over $k$ and finite Morley rank groups in $\dcf$ definable over $k$ (see~\cite[Fact~2.6]{kowalskipillay}).

\begin{fact}
\label{simple}
Suppose $(G,s)$ is a connected $D$-group over the constants.
\begin{itemize}
\item[(a)]
The centre $Z(G)$ is a normal $D$-subgroup of $G$ over the constants, and the quotient $G/Z(G)$ is an isotrivial $D$-group.
\item[(b)]
Let $H$ be the algebraic subgroup of points in $Z(G)$ where $s$ agrees with the zero section.
Then $Z(G)/H$ is an isotrivial $D$-group.
\end{itemize}
\end{fact}

\begin{proof}
For a proof that $Z(G)$ is a $D$-subgroup see~\cite[2.7(iii)]{kowalskipillay}.
That $G/Z(G)$ is isotrivial was originally proved by Buium~\cite{Buium} in the centerless case, and then  generalised by Kowalski and Pillay in~\cite[2.10]{kowalskipillay}.

For part~(b), note first of all that $H$ is a $D$-subgroup of $Z(G)$ by definition; the zero section does map to the tangent bundle of $H$.
Now, it suffices to show that $Z^\circ/H^\circ$ is isotrivial where $Z^\circ$ is the connected component of identity of $Z(G)$ and $H^\circ:=Z^\circ\cap H$.
Let $\mathcal Z:=\big(Z^\circ,s\big)^\sharp(K)$, the subgroup of $D$-points of $Z$.
Then $(H^\circ,s)^\sharp(K)=\mathcal Z(K^\delta)$, the $\delta$-constant points of $\mathcal Z$.
These are now commutative $\delta$-algebraic groups.
As the $\sharp$ functor is an equivalence of categories, isotriviality of $Z^\circ/H^\circ$ will follow once we show that $\mathcal Z/\mathcal Z(K^\delta)$ is definably isomorphic (over some parameters) to $(K^\delta)^n$ for some~$n$. %see~\cite[Fact~2.6]{kowalskipillay}
Because $Z^\circ$ is a connected commutative algebraic group over the constants, there exists a $\delta$-algebraic group homomorphism $\ell d:Z^\circ\to L(Z^\circ)$ over~$k_0$, where $L(Z^\circ)$ is the Lie algebra of~$Z^\circ$, the tangent space at the identity.
This homomorphism is called the {\em logarithmic derivative} and is defined as
$$\ell d(X)=\nabla(X)\cdot(s(X))^{-1}$$
where the operations occur in $TZ^{\circ}$. One can check that $\ell d$ is surjective with kernel $Z^\circ(K^\delta)$ (a proof appears in~\cite[$\S$3]{marker}, see also~\cite[$\S$V.22]{kolchin}).
So $\mathcal Z/\mathcal Z(K^\delta)$ is definably isomorphic to a $\delta$-algebraic subgroup, $\mathcal F$, of $L(Z^\circ)$.
Since $L(Z^\circ)$ is a vector group, $\mathcal F$ is a finite-dimensional $K^\delta$-vector subspace (see, for example,~\cite[Fact~1.3]{pillay96}), and hence definably isomorphic over a basis to some~$(K^\delta)^n$.
\end{proof}

Suppose $(G,s)$ is a $D$-group over a $\delta$-field $k$ and $(H,s)$ is a $D$-subgroup over $k$
Even when $H$ is not normal, it makes sense to consider the quotient space $G/H$ as an algebraic variety, and $s$ will induce on $G/H$ the natural structure of a $D$-variety $(G/H,\bar s)$ over $k$, in such a way that the quotient map $\pi:G\to G/H$ is a $D$-morphism.
See~\cite[Fact~2.7(ii)]{kowalskipillay}, for details.
Now if $\alpha$ is a $D$-point of $(G/H,\bar s)$, then the fibre $\pi^{-1}(\alpha)$ is a $D$-subvariety over $k(\alpha)$; and for $\beta$ a $D$-point of this fibre we have $\pi^{-1}(\alpha)=\beta+H$.
So each fibre $\big(\pi^{-1}(\alpha),s)$ is isomorphic to $(H,s)$ over $k(\beta)$.
One could develop in this context the notion of ``$D$-homogeneous spaces".

Using Fact~\ref{simple} we obtain the following highly restrictive property on the structure of $D$-subvarieties of $D$-groups over the constants.

\begin{proposition}
\label{ci}
Suppose $(G,s)$ is a connected $D$-group over $k_0\subseteq K^\delta$.
If $k$ is any $\delta$-field extension of $k_0$ and $W$ is any irreducible $D$-subvariety of $G$ over $k$, then $W$ is compound isotrivial in at most $3$ steps.

In particular, if $W$ is $\delta$-rational then it is $\delta$-locally closed.
\end{proposition}

\begin{proof}
Consider the normal sequence of $D$-subgroups
$$G\rhd Z(G)\rhd H\rhd 0$$
where $Z(G)$ is the centre of $G$ and $H$ is the algebraic subgroup of points in $Z(G)$ where $s$ agrees with the zero section.
Consider the corresponding sequence of irreducible $D$-varieties and $D$-morphisms over $k_0$
$$
\xymatrix{
G\ar[r]^{\pi_0\ \ \ }&G/H\ar[r]^{\pi_1\ \ }&G/Z(G)\ar[r]^{\pi_2}&0}.
$$
Since $G/Z(G)$, $Z(G)/H$, and $H$ are isotrivial -- the first two by~Fact~\ref{simple} and the last as $s|_H$ is the zero section -- this exhibits $G$ as compound isotrivial in three steps.
We can then obtain the same result for any irreducible $D$-subvariety of $G$ by using the fact that any element of $(G,s)^\sharp(K)$ is a product of two generic elements.
Alternatively we can argue as follows, keeping in mind that every $D$-subvariety of an isotrivial $D$-group is itself isotrivial.

If $W\subseteq G$ is an irreducible $D$-variety over $k$, then we get a sequence of dominant $D$-morphisms
$$
\xymatrix{
W\ar[r]^{f_0}&W_1\ar[r]^{f_1}&W_{2}\ar[r]^{f_2}&0}
$$
where $W_1\subseteq G/H$ is the Zariski closure of $\pi_0(W)$, and $W_2\subseteq G/Z(G)$ is the Zariski closure of $\pi_1(W_1)$, and the $f_i$ are the appropriate restrictions of the $\pi_i$.
Then $W_2$ is isotrivial as it is a $D$-subvariety of $G/Z(G)$.
If $\alpha$ is a $D$-point of $W_2$ then $f_1^{-1}(\alpha)$ is a $D$-subvariety of $\pi_1^{-1}(\alpha)$ which is isomorphic as a $D$-variety to $Z(G)/H$.
So the fibres of $f_1$ over $D$-points are all isotrivial $D$-subvarieties of $W_1$.
If $\beta$ is a $D$-point of $W_1$ then $f_0^{-1}(\beta)$ is a $D$-subvariety of $\pi_0^{-1}(\beta)$ which is isomorphic as a $D$-variety to $H$.
So the fibres of $f_0$ over $D$-points are all isotrivial.
Hence $W$ is compound isotrivial in $3$ steps.

The ``in particular" clause is by Proposition~\ref{compoundisotrivial}.
\end{proof}

We have now proved Theorem~A of the introduction:

\begin{corollary}
\label{dgroupDME}
If $k\subseteq K^\delta$ then every $D$-group over $k$ satisfies $\delta$-$\dme$.
\end{corollary}
\begin{proof}
Suppose $(G,s)$ is a $D$-group over $k$ and $W$ is an irreducible $D$-subvariety of $G$ over $k^{\alg}$.
Then, over $k^{\alg}$, it is isomorphic to an irreducible $D$-subvariety of the connected component of identity, $G_0$.
Applying Proposition~\ref{ci} to $G_0$, we have that $W$ is compound isotrivial.
So every irreducible $D$-subvariety of $G$ over $k^{\alg}$ is compound isotrivial.
The $\delta$-$\dme$ now follows from Corollary~\ref{deduce-diffDME}.
\end{proof}

\medskip
\subsection{Differential Hopf algebras}
In this section we give equivalent algebraic formulations of the $\delta$-$\dme$ and our results so far.
This will help us make the connection to the classical $\dme$, which is about noncommutative associative algebras and as such does not have a direct geometric formulation.

We restrict our attention to the case when $\delta$ is trivial on the base field $k$.

As explained in~$\S$\ref{dvar-section}, the standard geometry-algebra  duality which assigns to a variety its co-ordinate ring, induces an equivalence between the category of $k$-irreducible affine $D$-varieties $(V,s)$ and that of differential rings $(R,\delta)$ where $R$ is an affine $k$-algebra and $\delta$ is a $k$-linear derivation.
This equivalence associates to a $k$-irreducible $D$-subvariety of $V$ a prime $\delta$-ideal of $R$.
Using this dictionary, we can easily translate the geometric Definition~\ref{deltadme-geometric}, in the case when $k\subseteq K^\delta$, into the following algebraic counterpart.

\begin{definition}[$\delta$-$\dme$ for affine differential algebras]
\label{deltadme-rings} 
Suppose $R$ is an affine $k$-algebra equipped with a $k$-linear derivation $\delta$.
We say that $(R,\delta)$ {\em satisfies the $\delta$-$\dme$} if for every prime $\delta$-ideal $P$ of $R$, the following conditions are equivalent:
\begin{itemize}
\item[(i)]
$P$ is {\em $\delta$-primitive}:
There exists a maximal ideal $\mathfrak m$ of $R$ such that $P$ is maximal among the prime $\delta$-ideals contained in $\mathfrak m$.
\item[(ii)]
$P$ is {\em $\delta$-locally-closed}:
The intersection of all the prime $\delta$-ideals of $R$ that properly contain $P$ is a proper extension of $P$.
\item[(iii)]
$P$ is {\em $\delta$-rational}: $\operatorname{Frac}(R/P)^\delta$ is contained in $k^{\alg}$.
\end{itemize}
\end{definition}

The algebraic counterpart of an affine algebraic group $G$ over $k$ is the commutative Hopf $k$-algebra $R=k[G]$, where the group law $G\times G\to G$ induces a co-product $\Delta:R\to R\otimes_kR$.
So what is the algebraic counterpart of a $D$-group $(G,s)$ over $k$?
The following lemma says that it is a {\em differential Hopf $k$-algebra}, a commutative Hopf $k$-algebra $R$ equipped with a $k$-linear derivation $\delta$ that commutes with the coproduct, where $\delta$ acts on $R\otimes_kR$ by $\delta(r_1\otimes r_2)=\delta r_1\otimes r_2+r_1\otimes \delta r_2$.

\begin{lemma}
\label{geomdhopf}
Suppose $k\subseteq K^\delta$ and let $(G,s)$ be a $D$-variety defined over~$k$ such that $G$ is a connected affine algebraic group.
Let $\delta$ on $R=k[G]$ be the corresponding $k$-linear derivation.
Then $s:G\to TG$ is a group homomorphism if and only if $\delta$ commutes with the coproduct.
\end{lemma}

\begin{proof}
Unravelling the fact that $s$ induces the derivation $\delta$ on $k[G]$ and that the group operation $m:G\times G\to G$ induces the coproduct $\Delta$ on $k[G]$, we have that for all $f\in k[G]$,
\begin{equation}
\label{dhopf}
\Delta(\delta f)=\delta\Delta(f)\iff df(s(m(y,z)))=d(f\circ m)(s(y),s(z))
\end{equation}
where $(y,z)$ are coordinates for $G\times G$.
But
$$d(f\circ m)(s(y),s(z))=df\circ dm(s(y),s(z))=df(s(y)*s(z)),$$
where $*$ denotes the group operation $dm:TG\times TG\to TG$.
And so the right hand side of~(\ref{dhopf}) is equivalent to 
$$df(s(m(y,z)))=df(s(y)*s(z)).$$
But this, asserted for all $f\in k[G]$, is equivalent to $s(m(y,z))=s(y)*s(z)$, i.e., that $s$ is a group homomorphism.
\end{proof}

%\begin{remark}If no restriction is imposed on the coefficient field (such as $k\subseteq K^\delta$) then for general reasons examples of $D$-varieties whose co-ordinate rings do not satisfy the $\delta$-$\dme$ abound.Let $k_0$ be a field of constants. It is known that $k_0$ has no minimal differential closure, see for example \cite[Corollary~6.4]{Marker2}. Hence, there are $k \subsetneq K$, both differential closures of $k_0$. It follows that the constants of $k$ and $K$ are $k_0^{\alg}$. Let $a\in K\setminus k$. Then $a$ is differentially algebraic over $k$, and so the differential field generated by $a$ over $k$ agrees with $k(V)$ for some $k$-irreducible $D$-variety $(V,s)$. We have that $k(V)^\delta\subseteq k_0^{\alg}$; however, since $k$ is differentially closed, the type $tp(a/k)$ cannot be isolated. This implies that $V$ does not have a maximum proper $D$-subvariety over $k$, by Fact~\ref{equivisolated}. Hence, $(k[V],\delta)$ does not satisfy $\delta$-$\dme$.\end{remark}

In other words, the study of connected affine $D$-groups over the constants is the same thing as the study of affine differential Hopf $k$-algebras.
So our Theorem~A becomes:

\begin{theorem}
\label{dgroupDME-ring}
Every commutative affine differential Hopf algebra over a field of characteristic zero satisfies $\delta$-$\dme$.
\end{theorem}

\begin{proof}
By Lemma~\ref{geomdhopf} our differential Hopf algebra is of the form $k[G]$ for some connected affine $D$-group $(G,s)$ with $k\subseteq K^\delta$.
By Corollary~\ref{dgroupDME}  $(G,s)$ satisfies the $\delta$-$\dme$.
So $(k[G],\delta)$ satisfies $\delta$-$\dme$.
\end{proof}

\bigskip
\section{Twisting by a group-like element}
\label{atwist}

\noindent
As it turns out, the application to the classical Dixmier-Moeglin problem that we have in mind, and that will be treated in $\S$\ref{classical}, requires a generalisation of Theorem~\ref{dgroupDME-ring}.
Instead of working with differential Hopf algebras, we need to consider Hopf algebras equipped with derivations that do not quite commute with the coproduct.
Suppose $R$ is a commutative affine Hopf $k$-algebra.
We will use Sweedler notation\footnote{Recall that in Sweedler notation $\sum r_1\otimes r_2$ is used to denote an expression of the form $\sum_{j=1}^d r_{j,1}\otimes r_{j,2}$. We will use Sweedler notation throughout, hopefully without confusion.}
and write $\Delta(r)=\sum r_1\otimes r_2$ for any $r\in R$.
Now, for a $k$-linear derivation $\delta$ to commute with $\Delta$ on $R$ means that for all $r\in R$,
$$\Delta(\delta r)=\sum\delta r_1\otimes r_2+r_1\otimes\delta r_2.$$
We wish to weaken this condition by asking instead simply that there exists some $a\in R$ satisfying $\Delta(a)=a\otimes a$ -- that is, $a$ is a {\em group-like} element of $R$ -- such that for all $r\in R$,
\begin{equation}
\label{a-coder}
\Delta(\delta r)=\sum\delta r_1\otimes r_2+ar_1\otimes\delta r_2.
\end{equation}
That is, we ask $\delta$ to be what Panov~\cite{panov} calls an {\em $a$-coderivation}.
We wish to prove:

\begin{theorem}
\label{twistdgroupDME-ring}
Suppose $k$ is a field of characteristic zero, $R$ is a commutative affine Hopf $k$-algebra, and $\delta$ is a $k$-linear derivation on $R$ that is an $a$-coderivation for some group-like $a\in R$.
Then $(R,\delta)$ satisfies the $\delta$-$\dme$.
\end{theorem}

When $a=1$ this is just the case of affine differential Hopf $k$-algebras, and hence is dealt with by Theorem~\ref{dgroupDME-ring}.
The general case requires some work. Throughout this section $k$ is a fixed field of characteristic zero.

Let us begin with a geometric explanation of what this twisting by a group-like element means.
First of all, we have $R=k[G]$ for some connected affine algebraic group $G$ over $k$, with the coproduct $\Delta$ on $R$ induced by the group operation on $G$, and the derivation $\delta$ on $R$ induced by a $D$-variety structure $s:G\to TG$.
Note that $\delta$ being a $k$-derivation implies that $k\subseteq K^\delta$ and so $\tau G=TG$.
Now, as $G$ is an affine algebraic group, we may assume it is an algebraic subgroup of $\operatorname{GL}_n$, so that $TG\subseteq G\times\operatorname{Mat}_n$.
Writing $s=(\id,\overline s)$ where $\overline s:G\to\operatorname{Mat}_n$, we want to express as a property of $\overline s$ what it means for $\delta$ to be an $a$-coderivation.
That $a\in R$ is group-like means that $a:G\to\mathbb G_m$ is a homomorphism of algebraic groups.

\begin{lemma}
Suppose $G\subseteq\gln$ is a connected affine algebraic group over $k$, $a:G\to\mathbb G_m$ is a homomorphism, and $s=(\id,\overline s):G\to TG\subseteq G\times\matn$ is a $D$-variety structure on $G$ over $k$.
Then the corresponding $k$-linear derivation $\delta$ on $k[G]$ is an $a$-coderivation if and only if
\begin{equation}
\label{geom-a-coder}
\overline s(gh)=\overline s(g)h+a(g)g\overline s(h)
\end{equation}
for all $g,h\in G$, where all addition and multiplication is in the sense of matrices.
\end{lemma}

\begin{proof}
Note that for $r\in k[G]$, $\Delta(\delta r)\in k[G\times G]$ is given by
$$\Delta(\delta r)(g,h)=d_{gh}r\big(\overline s(gh)\big)$$
for all $g,h\in G$, where $dr:TG\to\mathbb A^2$ is the differential of $r:G\to\mathbb A^1$.
On the other hand,  writing $\Delta(r)=\sum r_1\otimes r_2$ we have
\begin{eqnarray*}
\sum (\delta r_1\otimes r_2 +a r_1\otimes \delta r_2)(g,h)
&=&
\sum d_gr_1(\overline s(g))\, r_2(h)+a(g)r_1(g)\, d_hr_2(\overline s(h))\\
&=&
d_{(g,h)}\left(\sum r_1\otimes r_2\right)(\overline s(g), a(g)\overline s(h))\\
&=&
d_{(g,h)}(\Delta r)(\overline s(g), a(g)\overline s(h))
\end{eqnarray*}
where the second equality uses the fact that $a(g)$ is a scalar.
Now, as an element of $k[G\times G]$, $\Delta(r)=r\circ m$ where $m:G\times G\to G$ is the restriction of matrix multiplication on $\gln$.
Note that when we differentiate matrix multiplication we get $d_{(g,h)}m(A,B)=Ah+gB$, for all $g,h\in \gln$ and $A,B\in \matn$.
Hence,
\begin{eqnarray*}
\sum (\delta r_1\otimes r_2 +a r_1\otimes \delta r_2)(g,h)
&=&
d_{(g,h)}(\Delta r)(\overline s(g), a(g)\overline s(h))\\
&=&
d_{gh}r\circ d_{(g,h)}m(\overline s(g), a(g)\overline s(h))\\
&=&
d_{gh}r(\overline s(g)h+ga(g)\overline s(h))\\
&=&
d_{gh}r\big(\overline s(g)h+a(g)g\overline s(h)\big).\\
\end{eqnarray*}
Hence, $\delta$ being an $a$-coderivation, that is equation~\eqref{a-coder}, is equivalent to
$$d_{gh}r\big(\overline s(gh)\big)=d_{gh}r\big(\overline s(g)h+a(g)g\overline s(h)\big)$$
for all $r\in k[G]$.
But this implies
$$\overline s(gh)=\overline s(g)h+a(g)g\overline s(h)$$
as desired.
\end{proof}

\begin{definition}
When $G$ is an affine algebraic group and $(G,s)$ is a $D$-variety structure such that~\eqref{geom-a-coder} holds, we will say that $(G,s)$ is an {\em $a$-twisted $D$-group}.
\end{definition}

The following family of examples of $2$-dimensional twisted $D$-groups will play an important role in the proof.

\begin{example}
\label{atwistedex}
Let $c\in k$ be a parameter.
Let $R=k[x,\frac{1}{x}, y]$ with $\delta$ the $k$-linear derivation induced by $\delta(x)=xy$ and $\delta(y)=\frac{y^2}{2}+c(1-x^2)$.
Note that $R$ is the co-ordinate ring of the algebraic subgroup $E\leq\operatorname{GL}_2$ made up of matrices of the form
\[\left( \begin{array}{cc}
x&y\\ 0&1\\ \end{array} \right),\]
and hence is a commutative affine Hopf $k$-algebra.
We denote by $(E,t_c)$ the $D$-variety structure on $E$ induced by $\delta$.
Writing $t_c=(\id,\overline t_c)$, we have
\[\overline t_c\left( \begin{array}{cc}
a&b\\ 0&1\\ \end{array} \right)
=\left( \begin{array}{cc}
ab&\frac{b^2}{2}+c(1-a^2)\\ 0&0\\ \end{array} \right).\]
Now a straightforward computation shows that
\begin{eqnarray*}
\overline t_c\left(
\left( \begin{array}{cc}
a&b\\ 0&1\\ \end{array} \right)
\left( \begin{array}{cc}
a'&b'\\ 0&1\\ \end{array} \right) \right)
&=&
\left(\begin{array}{cc}
a^2a'b'+aa'b&\frac{a^2(b')^2}{2}+abb'+\frac{b^2}{2}+c(1-(aa')^2)\\ 0&0\\ \end{array} \right)\\
%&=&
%\left( \begin{array}{cc}
%aa'b&abb'+\frac{b^2}{2}\\ 0&0\\ \end{array}\right) +
%\left( \begin{array}{cc}
%a^2a'b'&\frac{a^2(b')^2}{2}\\ 0&0\\
%\end{array}\right) \\
&=&
\overline t_c\left( \begin{array}{cc}
a&b\\ 0&1\\ \end{array} \right)
\left( \begin{array}{cc}
a'&b'\\ 0&1\\ \end{array} \right) +
a \left( \begin{array}{cc}
a&b\\ 0&1\\ \end{array} \right)
\overline t_c\left(
\begin{array}{cc}
a'&b'\\ 0&1\\ \end{array} \right).
\end{eqnarray*}
That is, $(E,t_c)$ is an $x$-twisted $D$-group.
(Note that $x\in R$ is group-like.)
Note that since $(E,t_c)$ is not a $D$-group, we cannot use Theorem~\ref{dgroupDME-ring} to deduce the $\delta$-$\dme$.
%However, as Ruizhang Jin pointed out to us, $(E,t)$, and indeed every $D$-subvariety of it, is isotrivial. To see this, note first of all that the subvariety of $E$ defined by $y=0$ is a $D$-subvariety on which $t$ is the zero section, and hence that $D$-subvariety is isotrivial. If $W\subseteq E$ is any other $D$-subvariety, then it has a generic $D$-point $(a,b)\in(E,t)^\sharp(K)$ with $b\neq 0$, and one can compute that $\delta^3\left(\frac{1}{a}\right)=\delta^2\left(\frac{1}{b}\right)=0$, so that the type of $(a,b)$ is internal to the constants, and hence $(W,t)$ is isotrivial.Since every $D$-subvariety of $(E,t)$ is isotrivial, Corollary~\ref{deduce-diffDME} implies that $(E,t)$, and hence $(R,\delta)$, satisfies $\delta$-$\dme$ after all. The isotriviality of all the $D$-subvarieties of $(E,t)$ says more and will be of use.
However, since the Krull dimension is two, $(E,t_c)$ does satisfy the $\delta$-$\dme$ (see Proposition~\ref{dmeleq2}).
\end{example}

Our strategy for proving Theorem~\ref{twistdgroupDME-ring} is to show that every $a$-twisted $D$-group over the constants admits the example described above as an image, with each fibre having the property that every $D$-subvariety is compound isotrivial.
From the $\delta$-$\dme$ for $(E,t_c)$, together with our earlier work around compound isotriviality and maximum proper $D$-subvarieties, we will then be able to conclude that every $a$-twisted $D$-group satisfies the $\delta$-$\dme$.

To relate an arbitrary $a$-twisted $D$-group to one of those considered in Example~\ref{atwistedex}, we will require the following proposition, whose proof is rather technical, and for which it would be nice to give a conceptual explanation.

\begin{proposition}
\label{magic}
Suppose $R$ is a commutative affine Hopf $k$-algebra, and $\delta$ is a $k$-linear derivation on $R$ that is an $a$-coderivation for some group-like $a\in R$.
Then for some $c\in k$ we have $a\delta^2 a=\frac{3}{2}(\delta a)^2+c(a^2-a^4)$.
\end{proposition}

We delay the proof of this proposition until we have established the preliminary Lemmas~\ref{deltaumeasy} and~\ref{lem: fg} below, for which we fix a commutative affine Hopf $k$-algebra, $R$, equipped with a $k$-linear derivation, $\delta$, such that $\delta$ is also an $a$-coderivation for some group-like $a\in R$.
As $a$ is group-like it is invertible in $R$.
We fix the following sequence of elements in $R$:
\begin{eqnarray*}
u_0&:=&a\\
u_1&:=&\frac{\delta a}{a}\\
u_2&:=&\delta u_1-\frac{u_1^2}{2}\\
u_m&:=&\delta(u_{m-1}) \ \ \ \ \text{for $m\ge 3$.}
\end{eqnarray*}
Note that the desired identity $a\delta^2 a=\frac{3}{2}(\delta a)^2+c(a^2-a^4)$ is equivalent to $u_2=c(1-a^2)$; this is our eventual aim.

\begin{lemma} 
\label{deltaumeasy}
For all $m\geq 1$, we have
$$\Delta(u_m) = u_m\otimes 1 + a^m\otimes u_m  +  \sum_{j=2}^{m-1}c_{j,m}a^ju_{m-j}\otimes u_j  + \sum f_i\otimes g_i$$
where the $c_{j,m}$ are positive (nonzero) integers, the $f_i\in (u_1,\dots,u_{m-1})^2k[u_0,\dots,u_{m-1}]$, and the $g_i\in k[u_0,\dots,u_{m-1}]$.
\end{lemma}
\begin{proof}
We can compute the coproducts of the elements $u_0, u_1,\ldots $ using the fact that $a=u_0$ is group-like and $\delta$ is an $a$-coderivation:
\begin{eqnarray*}
\Delta(u_0)&=&a\otimes a\\
\Delta(u_1)&=&(\frac{1}{a}\otimes\frac{1}{a})(\delta a\otimes a+a^2\otimes\delta a)\\
&=& u_1\otimes 1+a\otimes u_1\\
\Delta(u_2)&=&\delta u_1\otimes 1+\delta a\otimes u_1+a^2\otimes\delta u_1-\frac{u_1^2}{2}\otimes 1-au_1\otimes u_1-a^2\otimes \frac{u_1^2}{2}\\
&=&
u_2\otimes 1+a^2\otimes u_2
\end{eqnarray*}
Then for $m=1,2$, the conclusion of the statement of the lemma follows from the above computations with $f_i=g_i=0$ and the middle sum being empty.
Now one computes $\Delta(u_{m+1})=\Delta(\delta u_m)$ for $m\geq 2$, using the inductively given expression for $\Delta(u_m)$ and the fact that $\delta$ is an $a$-coderivation.
The rest is a straightforward brute force computation that we leave to the reader.
\end{proof}

\begin{lemma}
\label{lem: fg}
There exist $n\ge 1$, a polynomial $P\in k[u_0,\ldots ,u_{n-1}]$ and some $r\ge 0$ such that
$$u_n = \frac{P(u_0,\ldots ,u_{n-1})}{u_0^r}.$$
\end{lemma}
\begin{proof}
Since $R$ is finitely generated as a $k$-algebra, this sequence $(u_m)$ cannot be algebraically independent over $k$.
Choose $n$ minimal such that $(u_0,\dots,u_n)$ is algebraically dependent over $k$.
Note that if $n=0$ then $a=u_0$ is a constant and so $u_1=u_2=0$ by definition.
So we will assume that $n>0$.

So there is some $d\geq 1$ such that
\begin{equation}
u_n^d +\sum_{i<d} A_i(u_0,\ldots ,u_{n-1}) u_n^i = 0,
\label{eq: u}
\end{equation} with $A_0,\ldots ,A_{d-1}$ rational functions over $k$.
We may assume that $d$ is minimal.
Our first step is to show that $d=1$.

Since $R=k[G]$ for some connected affine algebraic group $G$ over $k$, we have that $R\otimes R$ is a domain. Indeed, $R\otimes R=k[G\times G]$ and $G\times G$ is a connected affine algebraic group.
We can thus work inside the fraction field of $R\otimes R$. Let $F$ be the subfield which is the fraction field of $k[u_0,\dots,u_{n-1}]\otimes_k k[u_0,\dots,u_{n-1}]$.
Note that by the minimality of $d$, $\{1,u_n,\dots,u_n^{d-1}\}$ is linearly independent over $k(u_0,\dots,u_{n-1})$, from which it follows that $\{u_n^i\otimes u_n^j:0\leq i,j<d\}$ is linearly independent over $F$.
Applying $\Delta$ to both sides of Equation (\ref{eq: u}), Lemma~\ref{deltaumeasy} gives us that
$$(u_n\otimes 1 + a^n \otimes u_n)^d \in \sum_{i<d} F\cdot(u_n\otimes 1+ a^n\otimes u_n)^i \subseteq \sum_{i+j<d} F \cdot(u_n^i\otimes u_n^j).$$
On the other hand, $u_n^d\otimes 1$ and 
$a^{nd}\otimes u_n^d$  are also in $\sum_{i+j<d} F \cdot(u_n^i\otimes u_n^j)$ by~(\ref{eq: u}).
It follows that 
$$\sum_{i=1}^{d-1} {d\choose i} a^{n(d-i)} u_n^i \otimes u_n^{d-i}\in \sum_{i+j<d} F\cdot(u_n^i\otimes u_n^j).$$
If $d>1$ then $u_n^{d-1}\otimes u_n$ appears with a nonzero coefficient on the left-hand side but with zero coefficient on the right-hand side.
This contradicts the $F$-linear independence of $(u_n^i\otimes u_n^j:0\leq i,j<d)$.

So $d=1$,  and we have that 
\begin{equation}
u_n = \frac{P(u_0,\ldots ,u_{n-1})}{Q(u_0,\ldots ,u_{n-1})},
\label{eq: uu} 
\end{equation}
for some relatively prime polynomials $P$ and $Q$ over $k$.
We aim to show that $Q$ is a monomial in $u_0$.

First we argue that $Q\otimes Q$ divides $\Delta(Q)$ in $S:=k[u_0,\dots,u_{n-1}]\otimes_kk[u_0,\dots,u_{n-1}]$, which note is the polynomial ring over $k$ in the variables $u_i\otimes1,1\otimes u_j$, and hence is a UFD.
Indeed,
\begin{eqnarray*}
\Delta(P)&=&\Delta(Q)\Delta(u_n) \ \ \ \ \ \ \ \text{by applying $\Delta$ to both sides of Equation (\ref{eq: uu})}\\
&=& \Delta(Q)(u_n\otimes 1 + a^n\otimes u_n +y) \ \ \ \ \ \ \ \text{by Lemma~\ref{deltaumeasy}, for some $y\in S$}\\
&=& \Delta(Q)\left((P/Q)\otimes 1 + a^n\otimes (P/Q) + y\right)
\end{eqnarray*}
 We can then multiply both sides by $1\otimes Q$ to see that
 $\Delta(Q)((P/Q)\otimes Q)\in S$.
 Hence, multiplying by $Q\otimes 1$, we see that $Q\otimes 1$ divides
 $\Delta(Q) (P\otimes Q)=\Delta(Q)(P\otimes 1)(1\otimes Q)$.
 Since $P$ and $Q$ are relatively prime,  $Q\otimes 1$ divides $\Delta(Q)$.
 A similar argument shows that $1\otimes Q$ divides $\Delta(Q)$.  Since we are working in a UFD and $1\otimes Q$ and $Q\otimes 1$ are relatively prime, we see that
 $Q\otimes Q$ divides $\Delta(Q)$, as desired.
 
Let $i\le n-1$ be the largest index for which $u_i$ appears in $Q$.
 Then we can write 
 $Q= \sum_{j=0}^M u_i^j Q_j(u_0,u_1,\ldots ,u_{i-1})$ with $M>0$ and $Q_M$ nonzero.
 So
 $$Q\otimes Q=(u_i^M\otimes u_i^M) (Q_M\otimes Q_M) + 
 \sum_{j,k<M} (u_i^j\otimes u_i^k) (Q_j\otimes Q_k)$$ 
 while, if $i\geq 1$, then
 $$\Delta(Q) = \sum_{j=0}^M \Delta(u_i)^j Q_j(\Delta(u_0),\ldots ,\Delta(u_{i-1})) = \sum_{\ell+m\le M} f_{\ell,m}(u_i^\ell \otimes u_i^m).$$
 where $f_{\ell,m}\in k[u_0,\dots,u_{i-1}]\otimes_kk[u_0,\dots,u_{i-1}]$, by Lemma~\ref{deltaumeasy}.
 (Note that Lemma~\ref{deltaumeasy} fails for $u_0$, so we are using that $i\geq 1$ in the above calculation.)
 But this contradicts $Q\otimes Q$ dividing $\Delta(Q)$, since $u_i^M\otimes u_i^M$ appears in the former while in the latter no $u_i^\ell \otimes u_i^m$ appears with $\ell,m\geq M$.
 So it must be that $i=0$ and we have shown that $Q$ is a polynomial in $u_0$.
 
Multiplying by a nonzero scalar if necessary, we may assume that $Q$ is in fact a polynomial in $u_0$ with leading coefficient $1$. Let $M$ denote the degree of $Q$. Then 
 $Q(u_0)\otimes Q(u_0)$ divides $\Delta(Q)=Q(u_0\otimes u_0)$ in $S$, recalling that $u_0=a$ is group-like.
 Since both $Q(u_0)\otimes Q(u_0)$ and $Q(u_0\otimes u_0)$ are polynomials of total degree $2M$ in the variables $u_0\otimes 1$ and $1\otimes u_0$ and since they both have leading coefficient $1$, we see that they must be the same.
 In particular, $Q(u_0)\otimes Q(u_0)$ is a polynomial in $u_0\otimes u_0$ with leading coefficient $1$, which implies that $Q(u_0)$ is of the form $u_0^r$.
 \end{proof}

\begin{proof}[Proof of Proposition~\ref{magic}]
Let $(R,\delta)$, $a$, and the $u_i$ be as above.
We need to show that $u_2=c(1-a^2)$ for some $c\in k$.
By Lemma \ref{lem: fg} we have that there is some $n\ge 1$, some $r\ge 0$ and some polynomial $P\in k[u_0,\ldots ,u_{n-1}]$ such that
\begin{equation}
u_n = \frac{P(u_0,\ldots ,u_{n-1})}{u_0^r},
\label{eq: uuu} 
\end{equation}
for some polynomial $P$ over $k$. 
Our first step is to show that $n\le 2$.

Let $S:=k[u_0,\dots,u_{n-1}]\otimes_kk[u_0,\dots,u_{n-1}]$.
Let
$$I:=(u_1,\dots,u_{n-1})^2k[u_0,\dots,u_{n-1}]$$
and consider the ideal of $S$ given by 
$$J:=I\otimes k[u_0,\dots,u_{n-1}]+k[u_0,\dots,u_{n-1}]\otimes I.$$
So these are the elements of $S$ in which each monomial has degree at least $2$ in either the variables $u_1\otimes1,\dots,u_{n-1}\otimes 1$, or in the variables $1\otimes u_1,\dots,1\otimes u_{n-1}$.
Using Lemma~\ref{deltaumeasy} we can compute that for $1\leq i,j,\ell\leq n-1$
\begin{equation}
\label{eq: Delta3}
\Delta(u_i u_ju_\ell) \in J.
\end{equation}
Moreover, for $1\leq i,j\leq n-1$, Lemma~\ref{deltaumeasy} gives
\begin{equation}
\label{eq: Delta}
\Delta(u_i u_j) = u_0^ju_i\otimes u_j+u_0^iu_j\otimes u_i  \mod J.
\end{equation}
Now, write the polynomial $P$ of equation~(\ref{eq: uuu}) as
$$P= P_0(u_0)+\sum_{i=1}^{n-1} P_i(u_0) u_i + \sum_{1\leq i\leq j\leq n-1}P_{i,j}(u_0)u_iu_j+H,$$
where $H$ is of degree at least three in $u_1,\dots,u_{n-1}$.
Applying $\Delta$ to both sides of equation~(\ref{eq: uuu}) we get $(u_0\otimes u_0)^r\Delta(u_n)
=
\Delta(P)$.
We therefore have:
\begin{equation}
\label{eq: cc}
\begin{aligned}
(u_0\otimes u_0)^r\Delta(u_n)
={}
& P_0(u_0\otimes u_0)+\sum_{i=1}^{n-1} P_i(u_0\otimes u_0)\Delta(u_i) \ +\\
& \sum_{i\leq j}P_{i,j}(u_0\otimes u_0)\Delta(u_iu_j)+\Delta(H).
\end{aligned}
\end{equation}
We claim that this forces $P_i=0$ for all $i=1,\dots,n-1$.
To prove this, note that by  Lemma~\ref{deltaumeasy} and equation~(\ref{eq: uuu}), both sides of equation~(\ref{eq: cc}) are elements of the polynomial ring $k(u_0\otimes 1, 1\otimes u_0)[u_1\otimes 1,\ldots ,u_{n-1}\otimes 1, 1\otimes u_1,\ldots ,1\otimes u_{n-1}]$.
We first compute, for both sides of~(\ref{eq: cc}), the coefficient of $u_i\otimes 1$.
On the right-hand side, using equations~(\ref{eq: Delta3}) and~(\ref{eq: Delta}), the only term that contributes is $P_i(u_0\otimes u_0)\Delta(u_i)$.
By Lemma~\ref{deltaumeasy}, that contribution is $P_i(u_0\otimes u_0)$.
On the left-hand side, using Lemma~\ref{deltaumeasy} and equation~(\ref{eq: uuu}), the coefficient of $u_i\otimes 1$ is $\displaystyle (u_0\otimes u_0)^r\left(\frac{P_i(u_0)}{u_0^{-r}}\otimes 1\right)$.
So $P_i(u_0)\otimes u_0^r=P_i(u_0\otimes u_0)$.
This forces $P_i=du_0^r$ for some $d\in k$.
On the other hand, comparing the coefficient of $1\otimes u_i$ on both sides of equation~(\ref{eq: cc}) we have that $u_0^{r+n}\otimes P_i(u_0)=P_i(u_0\otimes u_0)(u_0^i\otimes 1)$.
Plugging in $P_i=du_0^r$ we get that $d(u_0^{r+n}\otimes u_0^r)=d(u_0^{r+i}\otimes u_0^r)$.
As $i<n$, this forces $d=0$ and hence $P_i=0$.

Equation~(\ref{eq: cc}) therefore becomes
\begin{equation*}
\label{eq: uuudelta}
(u_0\otimes u_0)^r\Delta(u_n)=P_0(u_0\otimes u_0)+\sum_{1\leq i\leq j\leq n-1}P_{i,j}(u_0\otimes u_0)(u_0^ju_i\otimes u_j+u_0^iu_j\otimes u_i) \mod J.
\end{equation*}
Assume towards a contradiction that $n\geq 3$.
Then by Lemma~\ref{deltaumeasy} we must have $u_1\otimes u_{n-1}$ appearing in $\Delta(u_n)$ on the left with a nonzero coefficient.
So $P_{1,n-1}\neq 0$.
But then $P_{1,n-1}(u_0\otimes u_0)(u_0u_{n-1}\otimes u_1)$ appears on the right, while
it does not appear on the left since $u_{n-1}\otimes u_1$ does not appear in $\Delta(u_n)$ by Lemma~\ref{deltaumeasy}.

This contradiction proves that $n\leq 2$.
Suppose $n=1$.
Then equation~(\ref{eq: uuu}) says $u_1=\frac{P(u_0)}{u_0^r}$.
Applying $\Delta$ to both sides yields
$$P(u_0)\otimes u_0^r+u_0^{r+1}\otimes P(u_0)=P(u_0\otimes u_0)$$
which is only possible if $P_0=0$.
Hence $u_1=u_2=0$, as desired.

So we are left to consider the case when $n=2$.
Equation~(\ref{eq: uuu}) becomes
\begin{equation*}
u_2 = \frac{1}{u_0^r}\sum_{j=0}^M P_j(u_0) u_1^j
\end{equation*}
with $M\ge 1$, the $P_j$ are polynomials over $k$, and $P_M$ is nonzero.
Multiplying by $a^r$ (recall that $u_0=a$) and applying $\Delta$ gives
$$(a^r\otimes a^r)\left(u_2\otimes 1+ a^2\otimes u_2\right) = \sum_{j=0}^M P_j(a\otimes a) (u_1\otimes 1+ a\otimes u_1)^j$$
which we can write as 
$$\left(\sum_{j=0}^M P_j(a) u_1^j\otimes a^r + \sum_{j=0}^M a^{r+2} \otimes P_j(a) u_1^j\right) =  \sum_{j=0}^M P_j(a\otimes a) (u_1\otimes 1+ a\otimes u_1)^j.$$
Notice that if $M> 1$ then the right-hand side involves terms with $u_1^i \otimes u_1^j$ with $i,j\ge 1$ while the left-hand side does not, and so we cannot have equality.  Thus $M=1$.
Writing out the above equation with this in mind we get that
$$P_0(a)\otimes a^r + P_1(a)u_1\otimes a^r + a^{r+2}\otimes P_0(a)+a^{r+2}\otimes P_1(a)u_1$$ is equal to
$$P_0(a\otimes a) + P_1(a\otimes a)(u_1\otimes 1 + a\otimes u_1).$$
We look at this as an equation in $k[a\otimes 1,1\otimes a][u_1\otimes 1,1\otimes u_1]$. Then taking the coefficient of $u_1\otimes 1$ gives that
$$P_1(a)\otimes a^r = P_1(a\otimes a),$$
which can only occur if $P_1=da^r$ for some $d\in k$.  Then
computing the coefficient of $1\otimes u_1$ gives $a^{r+2}\otimes d a^r= d(a^{r+1}\otimes a^r)$, and so $d=0$.
Hence $P_1=0$.
Now taking the constant coefficient (regarding constants as being in $k[a\otimes 1,1\otimes a]$) gives that
$$P_0(a)\otimes a^r + a^{r+2}\otimes P_0(a) = P_0(a\otimes a).$$ 
Now write $P_0(t)=\sum_{j=0}^L p_j t^j$.   Then we have
$$\sum_{j=0}^L p_j (a^j\otimes a^r + a^{r+2}\otimes a^j - a^j \otimes a^j) = 0.$$
Notice that if $j\not\in \{r,r+2\}$ then we have that the coefficient of $a^j \otimes a^j$ on the left-hand side is equal to $p_j$ whereas the right-hand side is zero and so $p_j=0$.  It follows that $P_0(t) = p_r t^r + p_{r+2} t^{r+2}$.  
Then
$$0=\sum_{j=0}^L p_j (a^j\otimes a^r + a^{r+2}\otimes a^j - a^j \otimes a^j) =  p_r a^{r+2}\otimes a^r + p_{r+2} a^{r+2}\otimes a^r.$$
This forces $p_r=p_{r+2}$ and so we see $P_0(t)=c(t^r-t^{r+2})$ for some constant $c\in k$.

So $u_2=\frac{1}{u_0^r}\big(P_0(u_0)+P_1(u_0)u_1\big)=c(1-u_0^2)=c(1-a^2)$, as desired.
\end{proof}

The following is a geometric interpretation of the proposition.

\begin{proposition}
\label{toe}
Suppose $k\subseteq K^\delta$ and $(G,s)$ is an affine connected $a$-twisted $D$-group over $k$ where $a\in k[G]$ is group-like.
Then 
$$g\longmapsto\left(\begin{array}{cc} a(g) & \frac{\delta a(g)}{a(g)}\\ 0 & 1\\ \end{array} \right)$$
defines a homomorphism $\pi:G\to E$ where $E\leq\operatorname{GL}_2$ is the algebraic subgroup made up of matrices of the form
$\displaystyle \left( \begin{array}{cc}
x&y\\ 0&1\\ \end{array} \right)$.
Moreover, there exists some $c\in k$ such that if $(E,t_c)$ is the $a$-twisted $D$-group from Example~\ref{atwistedex}, then $\pi:(G,s)\to (E,t_c)$ is a $D$-morphism.
\end{proposition}

\begin{proof}
Recall that as $a\in k[G]$ is group-like, $a:G\to\mathbb G_m$ is a homomorphism of algebraic groups.
It follows immediately that $\pi$ is well-defined and does indeed map $G$ to $E$.
We check that it is a group homomorphism: given $g,h\in G$, note first of all that as $\Delta(a)=a\otimes a$ and $\delta$ is an $a$-coderivation we have $\Delta(\delta a)=\delta a\otimes a+a^2\otimes\delta a$ and so
\begin{equation}
\label{deltaagh}
\delta a(gh)
=\Delta(\delta a)(g,h)=\delta a(g) a(h)+a(g)^2\delta a (h).
\end{equation}
We can therefore compute
\begin{eqnarray*}
\pi(g)\pi(h) &=&
\left(\begin{array}{cc} a(g) & \frac{\delta a(g)}{a(g)}\\ 0 & 1\\ \end{array} \right)
\left(\begin{array}{cc} a(h) & \frac{\delta a(h)}{a(h)}\\ 0 & 1\\ \end{array} \right)\\
&=&
\left(\begin{array}{cc} a(gh) & a(g)\frac{\delta a(h)}{a(h)}+\frac{\delta a(g)}{a(g)}\\ 0 & 1\\ \end{array} \right)\\
&=&
\left(\begin{array}{cc} a(gh) & \frac{a(g)^2\delta a (h)+\delta a(g) a(h)}{a(gh)}\\ 0 & 1\\ \end{array} \right)\\
&=&
\left(\begin{array}{cc} a(gh) & \frac{\delta a(gh)}{a(gh)}\\ 0 & 1\\ \end{array} \right)\ \ \ \ \ \ \ \ \ \ \ \ \ \ \ \ \ \text{by~(\ref{deltaagh})}\\
&=&
\pi(gh)
\end{eqnarray*}
where we have used repeatedly that $a(gh)=a(g)a(h)$.  We note that we have not up until this point used the parameter $c\in k$; the reason for this is that the groups $E_c$ are isomorphic as algebraic groups.

It remains to show that $\pi$ is a $D$-morphism from $(G,s)$ to some $(E,t_c)$.
Let $c$ be as given by Proposition \ref{magic}.
It suffices to show that $\pi$ takes $D$-points to $D$-points.
That is, if $g\in (G,s)^\sharp(K)$ then $\left(\begin{array}{cc} a(g) & \frac{\delta a(g)}{a(g)}\\ 0 & 1\\ \end{array} \right)$ should be a $D$-point of $(E,t_c)$.
Writing $t_c=(\id,\overline t_c)$ we have that
\begin{eqnarray*}
\overline t_c\left(\begin{array}{cc} a(g) & \frac{\delta a(g)}{a(g)}\\ 0 & 1\\ \end{array} \right)
&=&
\left(\begin{array}{cc} \delta a(g) & \frac{\delta a(g)^2}{2a(g)^2}+c(1-a(g)^2)\\ 0 & 0\\ \end{array} \right)\ \ \ \text{ by Example~\ref{atwistedex}}\\
&=&
\left(\begin{array}{cc} \delta a(g) & \frac{a(g)\delta^2 a(g)-\delta a (g)^2-c(a(g)^2-a(g)^4)}{a(g)^2}+c(1-a(g)^2)\\ 0 & 0\\ \end{array} \right)\\
&=&
\delta\left(\begin{array}{cc} a(g) & \frac{\delta a(g)}{a(g)}\\ 0 & 1\\ \end{array} \right)
\end{eqnarray*}
where the penultimate step follows from Proposition~\ref{magic} telling us that $a\delta^2 a=\frac{3}{2}(\delta a)^2+c(a^2-a^4)$,   
and in the final equality we are using the fact that as $g$ is a $D$-point of $G$, $\delta(r(g))=(\delta r)(g)$ for all $r\in k[G]$.
This shows that $\pi(g)\in(E,t_c)^\sharp(K)$, as desired.
\end{proof}

We can now complete the proof of the theorem.

\begin{proof}[Proof of Theorem~\ref{twistdgroupDME-ring}]
We have already established that $(R,\delta)$ is the co-ordinate ring of an affine connected $a$-twisted $D$-group $(G,s)$ where $a\in k[G]$ is group-like. Here recall that $k\subseteq K^\delta$.
By Proposition~\ref{knownparts}, it suffices to show that every irreducible $\delta$-rational $D$-subvariety of $G$ over $k$ is $\delta$-locally-closed.

Let $\pi:(G,s)\to (E,t_c)$ be the $D$-morphism from Proposition~\ref{toe}.
We first show that every fibre of this map has the property that all its $D$-subvarieties, over arbitrary $\delta$-field extensions, are compound isotrivial.

Let us start with the fibre above the identity, that is, $H=\ker(\pi)$.
Since $\pi$ is a $D$-morphism, $(H,s)$ is a $D$-subvariety of $(G,s)$.
Here, by abuse of notation, we write $(H,s)$ instead of $(H,s\upharpoonright_H)$.
Since $\pi$ is an algebraic group homomorphism $H$ is an algebraic subgroup of $G$.
It follows that $(H,s)$ is an $a{\upharpoonright}_H$-twisted $D$-group also.
On the other hand, $a{\upharpoonright}_H=1$ by the definition of $\pi$.
So $(H,s)$ is an actual $D$-group.
By Proposition~\ref{ci}, every irreducible $D$-subvariety of $H$, over any $\delta$-field extension of $k$, is compound isotrivial.

What about other fibres of $\pi$ over $D$-points of $(E,t_c)$?
Any such fibre is a $D$-subvariety of $(G,s)$ of the form $Hg$, for some $g\in(G,s)^\sharp(K)$.
Since $(G,s)$ is not necessarily a $D$-group, the multiplication-by-$g$-on-the-right map, $\rho_g:G\to G$, is not necessarily a $D$-automorphism.
Nevertheless, when we restrict this map to $H$ we do get a $D$-isomorphism between $H$ and $Hg$.
To see this we need only check that $\rho_g$ takes $D$-points of $H$ to $D$-points of $Hg$.
Letting $h\in (H,s)^\sharp(K)$ we compute
\begin{eqnarray*}
\overline s(hg)
&=&\overline s(h)g+a(h)h\overline s(g) \ \ \ \ \text{by~(\ref{geom-a-coder})}\\
&=&\overline s(h)g+h\overline s(g) \ \ \ \ \ \ \ \ \ \ \text{as $a{\upharpoonright}_H=1$}\\
&=&\delta(h)g+h\delta(g)\ \ \ \ \ \ \ \ \ \ \text{as $h$ and $g$ are $D$-points}\\
&=&\delta(hg)\ \ \ \ \ \ \ \ \ \ \ \ \ \ \ \text{as $\nabla:G\to TG$ is a group homomorphism}
\end{eqnarray*}
as desired.
So $H$ and $Hg$ are $D$-isomorphic over $k(g)$.
It follows that every fibre of $\pi$ above a $D$-point has the property that all its $D$-subvarieties, over arbitrary $\delta$-field extensions, are compound isotrivial.

Now suppose that $V\subseteq G$ is an irreducible $\delta$-rational $D$-subvariety over $k$.
We need to prove that it has a maximum proper $D$-subvariety over $k$.
Let $W\subseteq E$ be the $D$-subvariety obtained by taking the Zariski closure of the image of $V$ under~$\pi$, and 
consider the dominant $D$-morphism $\pi{\upharpoonright}_V:(V,s)\to (W,t_c)$.
Since $k(W)\subseteq k(V)$, $W$ is also $\delta$-rational.
Since $(E,t_c)$ is of dimension two, it satisfies the $\delta$-$\dme$ by Proposition~\ref{dmeleq2}.
Hence $W$ has a maximum proper $D$-subvariety over $k$.
Next, let $\eta$ be a $k$-generic $D$-point of $W$ and consider the fibre $V_\eta$.
Note that $V_\eta$ is $\delta$-rational since $k(\eta)(V_\eta)=k(V)$.
But $V_\eta$ is a $D$-subvariety of the fibre of $\pi:(G,s)\to (E,t_c)$ above the $D$-point $\eta$, and hence as we have argued above, is compound isotrivial.
So, by Proposition~\ref{compoundisotrivial}, $V_\eta$ has a maximum proper $D$-suvariety over $k(\eta)$.
We have shown that both the image and the generic fibre have maximum proper $D$-subvarieties, and so by Lemma~\ref{equivisolated}, $(V,s)$ has a maximum proper $D$-subvariety over~$k$, as desired.
\end{proof}

\begin{remark}
In the end of above proof we could also have used the fact that $(E,t_c)$, while not in general isotrivial, is compound isotrivial in two steps.
This was observed by Ruizhang Jin, in whose PhD thesis this example will be worked out.
In any case, using the compound isotriviality of $(E,t_c)$ the above arguments actually give that every $D$-subvariety of $(G,s)$ is compound isotrivial (in at most five steps) from which it follows by Corollary~\ref{deduce-diffDME} that $(G,s)$ satisfies the $\delta$-$\dme$.
\end{remark}

\bigskip
\section{The $\dme$ for Ore extensions of commutative Hopf agebras}
\label{classical}

\noindent
We will now apply the results of the previous sections to the classical study of certain (noncommutative) Hopf agebras.
Recall that if $A$ is a noetherian associative algebra over a field $k$ of characteristic zero, then we say that the {\em Dixmier-Moeglin equivalence} ($\dme$) {\em holds for $A$} if for every (two-sided) prime ideal $P$ of $A$, the following are equivalent:
\begin{itemize}
\item[(i)]
$P$ is {\em primitive}: it is the annhilator of a simple left $A$-module.
\item[(ii)]
$P$ is {\em locally closed}: the intersection of all the prime ideals of $A$ that properly contain $P$ is a proper extension of $P$.
\item[(iii)]
$P$ is {\em rational}: the centre of the Goldie quotient ring\footnote{The Goldie quotient is an artinian ring of quotients for any prime noetherian ring that imitates the field of fractions construction for integral domains in the commutative case. See~\cite[Chapter 2]{mcconnell-robson} for details.} $\operatorname{Frac}(A/P)$ is an algebraic field extension of $k$.
\end{itemize}
Of course, for commutative algebras the $\dme$ always holds as the notions of primitive, locally closed and rational all coincide with maximal.

It is known that in any algebra that satisfies the Nullstellensatz %, in particular when $k$ is an uncountable algebraically closed field and $A$ is finitely generated over $k$, 
 locally closed implies primitive and primitive implies rational, see~\cite[II.7.16]{BrownGoodearl}.
Thus, the central question is: when does rational imply locally closed?
Certainly this is not always the case; even in finite Gelfand-Kirillov dimension a counterexample was found in~\cite{pdme}.
In~\cite{BellLeung} the $\dme$ was conjectured specifically about all Hopf algebras of finite Gelfand-Kirillov dimension.

We will show here that the $\dme$ holds for Hopf algebras that arise as certain twisted polynomial rings over commutative Hopf algebras.
Recall that if $R$ is a $k$-algebra equipped with an automorphism~$\sigma$ then a $k$-linear {\em $\sigma$-derivation} is a $k$-linear map $\delta$ satisfying the twisted Leibniz rule:
$$\delta(rs)=\sigma(r)\delta(s)+\delta(r)s.$$
Given $\sigma$ and $\delta$, the {\em Ore extension} of $R$, denoted by $R[x;\sigma,\delta]$ is the ring extension of $R$ with the property that it is a free left $R$-module with basis $\{x^n:n\geq 0\}$ and such that $xr=\sigma(r)x+\delta(r)$ for all $r\in R$.
We aim to prove the $\dme$ for Hopf algebras that arise as the Ore extensions of commutative Hopf algebras.
More precisely,

\begin{theorem}
\label{hodme}
Suppose $k$ is a field of characteristic zero and $R$ is a commutative affine Hopf $k$-algebra equipped with a $k$-algebra automorphism $\sigma$ and a $k$-linear $\sigma$-derivation $\delta$.
Assume that the Ore extension $A:=R[x;\sigma,\delta]$ admits a Hopf algebra structure extending that of $R$.
Then $A$ satisfies the $\dme$.
\end{theorem}
\noindent
This is Theorem~B2 of the introduction.
Its proof is preceded by a number of preliminaries.

\medskip
\subsection{Hopf Ore extensions}
In this section we prove a result (Corollary~\ref{ourbozz}, below) that severely restricts what $(R,\Delta,\sigma,\delta)$ can be if $A=R[x;\sigma,\delta]$ is to admit a Hopf algebra structure extending that on $R$.
Actually this was already done by Brown et al. in~\cite{bozz}, answering a question of Panov~\cite{panov}, in a more general context where $R$ is not necessarily commutative, but under the additional assumption on $A$ that
\begin{equation}
\label{bozzcondition}
\Delta(x)=a\otimes x+x\otimes b+v(x\otimes x)+w
\end{equation}
for some $a,b\in R$ and $v,w\in R\otimes_k R$.
When~(\ref{bozzcondition}) holds, possibly after a change of the variable $x$, Brown et al. call $R[x;\sigma,\delta]$ a {\em Hopf Ore} extension.
They ask if every Ore extension admitting a Hopf algebra structure extending that on $R$ is a Hopf Ore extension.
We prove that this is the case, in a strong way, when $R$ is commutative and affine (this is Theorem~C of the introduction):

\begin{theorem}
\label{deltax}
Suppose $k$ is an algebraically closed field of characteristic zero and $R$ is a commutative affine Hopf $k$-algebra equipped with a $k$-algebra automorphism $\sigma$ and a $k$-linear $\sigma$-derivation $\delta$.
If $R[x;\sigma,\delta]$ admits a Hopf algebra structure extending that of $R$ then, after a linear change of the variable $x$, 
$$\Delta(x)=a\otimes x+x\otimes b+w$$
for some $a,b\in R$, each of which is either $0$ or group-like, and some $w\in R\otimes_k R$.
\end{theorem}

\begin{proof}
Our starting point is~\cite[$\S2.2$, Lemma~1]{bozz} which says that if $R\otimes_kR$ is a domain (which is true here as $R$ is a commutative domain and $k$ is algebraically closed) then
\begin{equation}
\label{bozz2.2}
\Delta(x)=s(1\otimes x)+t(x\otimes 1)+v(x\otimes x)+w
\end{equation}
 where $s,t,v,w\in R\otimes_kR$.  We let $A=R[x;\sigma,\delta]$ and let $S$ denote the antipode of $A$.  
By making a substitution $x\mapsto x-\lambda$ for some $\lambda\in k$, we may assume that $\epsilon (x)=0$ where $\epsilon:A\to k$ is the counit.
This substitution does not change the form of $\Delta(x)$ given in~(\ref{bozz2.2}).

Our first goal is to show that $v=0$.
Recall that since $R$ is commutative it is of the form $R=k[G]$ for a connected affine algebraic group $G$.
We can therefore view $v\in R\otimes_kR$ as a regular function on $G\times G$.
We show first that $v(g^{-1},g)=0$ for all $g\in G$, and then that in fact $v=0$.

We now consider the antipode $S$. By \cite[Corollary 1]{Skryabin}, $S$ is bijective on $A$ and its restriction to $R$ is bijective on $R$.  Thus we can write
$$S(x) = a_0  + a_1 x+ \cdots + a_d x^d$$ for some $d\geq 1$ and $a_0,\ldots ,a_d\in R$ with $a_d\neq 0$.
Writing $m:A\otimes_kA\to A$ for the homomorphism induced by multiplication, we have the identity
$$m\circ (S\otimes \id)\circ \Delta(x) = \epsilon(x).$$
So, as $\epsilon(x)=0$, we may let $\mu=m\circ (S\otimes \id)$ and use~(\ref{bozz2.2}) to write
$$0 = \mu(s)x+S(x)\mu(t)+S(x)\mu(v)x+\mu(w).$$
Notice that $m\circ (S\otimes\id)(R\otimes R)\subseteq R$ and so if we look at the coefficient of $x^{d+1}$ on the right-hand side, we see that it is $a_d\sigma^d(\mu(v))$.
Since $R\otimes R$ is a domain and $a_d$ is nonzero and $\sigma$ is an automorphism, we see that $\mu(v)=m\circ (S\otimes\id)(v)=0$.
Geometrically, this means precisely that $v(g^{-1},g)=0$ for all $g\in G$.

Next we apply coassociativity, which tells us that $(\Delta\otimes\id)(\Delta(x))=(\id\otimes\Delta)(\Delta(x))$ in $R\otimes_kR\otimes_kR=k[G\times G\times G]$.
Writing this out using~(\ref{bozz2.2}), and equating the coefficients of $x\otimes x\otimes x$, yields
$$(\Delta\otimes\id)(v)\cdot(v\otimes 1)=(\id\otimes\Delta)(v)\cdot(1\otimes v).$$
Evaluating at $(g,h^{-1},h)$ for any fixed $g,h\in G$ we get
$$v(gh^{-1},h)v(g,h^{-1})= v(g,1_G)v(h^{-1},h)=0$$
where the final equality uses what we proved in the previous paragraph.
Now, if $v\neq 0$ then for a Zariski dense set of $(g,h)\in G\times G$, $v(g,h^{-1})\neq 0$.
But then for each such $(g,h)$ the above equation implies that $v(gh^{-1},h)=0$.
Hence, in fact, $v(gh^{-1},h)=0$ for all $(g,h)\in G\times G$.
As every element of $G\times G$ can be written in the form $(gh^{-1},h)$, we have shown that $v=0$.

We have thus proven that
\begin{equation}
\label{bozz2.2nov}
\Delta(x)=s(1\otimes x)+t(x\otimes 1)+w
\end{equation}
 for some $s,t,w\in R\otimes_kR$.
 
We claim now that either $t=0$ or $t=1\otimes b$ for some group-like $b\in R$.
We again apply coassociativity to $x$, this time using~(\ref{bozz2.2nov}) and equating the coefficients of $x\otimes 1\otimes 1$, to get
$$(\Delta\otimes\id)(t)\cdot (t\otimes 1)=(\id\otimes\Delta)(t)$$
The geometric interpretation is that 
\begin{equation}
\label{tfgh}
t(fg,h)t(f,g)=t(f,gh)
\end{equation}
 for all $f,g,h\in G$.
%Setting $f=1_G$ we get $t(g,h)t(1_G,g)=t(1_G,gh)$ for all $g,h\in G$.

Suppose $t(1_G,g_0)=0$ for some $g_0\in G$.
We show in this case that $t=0$.
Indeed, for all $h\in G$ we have $0=t(g_0,h)t(1_G,g_0)=t(1_G,g_0h)$, by~(\ref{tfgh}) with $f=1_G$.
Hence, $t(1_G,h)=0$ for all $h\in G$.
But then, by~(\ref{tfgh}) with $f=g^{-1}$, we get
$0=t(1_G,h)=t(g^{-1}g,h)t(g^{-1},g)=t(g^{-1},gh)$ 
for all $g,h\in G$.
As every element of $G\times G$ is of the form $(g^{-1},gh)$ for some $g,h\in G$, we have $t=0$, as desired.

Suppose on the contrary that $t(1_G,g)\neq 0$ for every $g\in G$.
Then $t(g,h)=\frac{t(1_G,gh)}{t(1_G,g)}$ is a never vanishing regular function on $G\times G$, and hence $t=\lambda t'$ where $\lambda\in k^*$ and $t':G\times G\to \mathbb G_m$ is an algebraic group homomorphism (see~\cite[Theorem~3]{rosenlicht}).
So $t'=b'\otimes b$ where $b',b\in R$ are group-like.
But then we have
$$\lambda b'(g)b(h)=t(g,h)=\frac{t(1_G,gh)}{t(1_G,g)}=\frac{\lambda b(gh)}{\lambda b(g)}=\frac{\lambda b(g)b(h)}{\lambda b(g)}$$
for all $g,h\in G$.
It follows that $b'=\lambda=1$ and $t=1\otimes b$, as desired.

A similar argument shows that in~(\ref{bozz2.2nov}) either $s=0$ or $s=a\otimes 1$ for some group-like $a\in R$.
This proves the theorem.
\end{proof}

\begin{remark}
It may be worth pointing out that our proof of Theorem~\ref{deltax} made no use of $\delta$.
We used only the properties of a Hopf algebra extension and the fact that $\sigma$ is injective, as well as the fact that every element of $A$ can be written as a left polynomial in $x$ over $R$.
\end{remark}

\begin{corollary}
\label{ourbozz}
Suppose $k$ is an algebraically closed field of characteristic zero and $R$ is a commutative affine Hopf $k$-algebra equipped with a $k$-algebra automorphism $\sigma$ and a $k$-linear $\sigma$-derivation $\delta$.
If $R[x;\sigma,\delta]$ admits a Hopf algebra structure extending that of $R$ then $(R,\Delta,\sigma,\delta)$ must satisfy the following two conditions.
\begin{itemize}
\item[(1)]
There exists $w\in R\otimes_k R$ and a group-like $a\in R$ such that, for all $r\in R$,
$$\Delta(\delta(r)) = \sum \left( \delta(r_1)\otimes r_2 + ar_1\otimes \delta(r_2)\right) + w\big(\Delta(r) -\Delta(\sigma(r))\big)$$
\item[(2)]
There is a character $\chi:R\to k$ such that for all $r\in R$,
$$\sigma(r)=\sum \chi(r_1)r_2 = \sum r_1 \chi(r_2).$$
\end{itemize}
In the above we are using Sweedler notation, writing $\Delta(r)=\sum r_1\otimes r_2$ for all $r\in R$.
\end{corollary}

\begin{proof}
Statements~(1) and~(2) are proven for Hopf Ore extensions in~\cite{bozz}.
Indeed, remembering that in our case $R$ is commutative, statement~(2) is just part~(i)(c) of the main theorem of~\cite{bozz} (see also Theorem~2.4(d) of~\cite{bozz}), and statement~(1) is the identity labelled~(21) in~\cite{bozz} which is asserted in part~(i)(d) of the main theorem.
So to prove the corollary it suffices to show that $A=R[x;\sigma,\delta]$ is a Hopf Ore extension, that is, after a change of variable $\Delta(x)$ has the form~(\ref{bozzcondition}) discussed above.
But Theorem~\ref{deltax} gives us an even stronger form for $\Delta(x)$.
\end{proof}

\begin{remark}
The main theorem of~\cite{bozz} also includes a converse; namely, assuming that $(R,\Delta,\sigma,\delta)$ satisfies~(1) and~(2), with $w\in R\otimes_kR$ satisfying two other identities, one can always extend in a natural way the Hopf algebra structure from $R$ to $R[x;\sigma,\delta]$.
This gives many examples to which our Theorem~\ref{hodme} will apply.
\end{remark}

\medskip
\subsection{The case when $\sigma$ is the identity}
When $\sigma=\id$ note that a $\sigma$-derivation is just a derivation.
In this case we write the Ore extension as $R[x;\delta]$; it is the skew polynomial ring in $x$ over $R$ where $xr=rx+\delta(r)$ for all $r\in R$.
Statement~(1) of Corollary~\ref{ourbozz} now says that if $R[x;\delta]$ admits a Hopf algebra structure extending that on $R$, then $\delta$ must have been an $a$-coderivation on $R$.
So Theorem~\ref{twistdgroupDME-ring} applies and we have that $(R,\delta)$ satisifes the $\delta$-$\dme$.
The following proposition relates the $\delta$-$\dme$ for $(R,\delta)$ to the $\dme$ for $R[x;\delta]$.

\begin{proposition}
\label{ddme-sdme}
Suppose $k$ is a field of characteristic zero and $R$ is a commutative affine $k$-algebra equipped with a $k$-linear derivation $\delta$.
If the $\delta$-rational prime $\delta$-ideals of $(R,\delta)$ are $\delta$-locally-closed, then the rational prime ideals of $R[x;\delta]$ are locally closed.

In particular %, assuming $k$ is an uncountable algebraically closed field, 
if the $\delta$-$\dme$ holds for $(R,\delta)$ then the $\dme$ holds for $R[x;\delta]$.
\end{proposition}

\begin{proof}
Suppose $P$ is a rational prime ideal of $R[x;\delta]$.
Let $I:=P\cap R$.
Then $I$ is a prime ideal of $R$ (see~\cite[Corollary to Lemma~2]{fisher}).
Moreover, $I$ is a $\delta$-ideal since if $a\in P\cap R$ then $\delta(a)=[x,a]\in P\cap R$.
It follows easily that $J:=IR[x;\delta]$ is an ideal of $R[x;\delta]$ that is contained in $P$ and that $R[x;\delta]/J\cong (R/I)[x;\delta]$ where we use $\delta$ to denote the induced derivation on $S:=R/I$.
Let $F=\Frac(S)$ be the field of fractions of $S$ and extend $\delta$ to $F$.
We claim that the $\delta$-constants of $F$ are all algebraic over $k$.
Indeed, note that if $f\in F$ with $\delta(f)=0$ then $f$ is a central element of $F[x;\delta]$.
We let $\tilde{P}$ denote the prime ideal in $S[x;\delta]$ corresponding to $P$ under the isomorphism $R[x;\delta]/J\cong S[x;\delta]$.
As $P\cap R=I$, we have that $\tilde{P}\cap S=0$, so that $\tilde{P}$ lifts to a prime ideal $P_0$ of $F[x;\delta]$.
The image of $f$ in $B:=F[x;\delta]/P_0$ is again a central element of $B$.
By construction $B$ is a localization of $S[x;\delta]/\tilde P\cong R[x;\delta]/P$ and thus passing to the full localization gives that $f$ is a central element of ${\rm Frac}(R[x;\delta]/P)$.
As $P$ is rational $f$ must be algebraic over~$k$.

We have shown that the prime $\delta$-ideal $I$ is $\delta$-rational.
By assumption it is therefore $\delta$-locally-closed.
Consequently, there is some $g\in R\setminus I$ such that every prime $\delta$-ideal of $R$ properly containing $I$ must contain~$g$.  

In order to prove that $P$ is locally closed it now suffices to show that whenever $Q\supsetneq P$ is prime then $Q\cap R\supsetneq P\cap R=I$. Indeed, if this is the case, then we have that
$\displaystyle g\in\bigcap\{Q\cap R:Q\supsetneq P\text{ prime}\}$.
Since $g\notin P$, we have in particular that $\displaystyle \bigcap\{Q:Q\supsetneq P\text{ prime}\}\neq P$.
That is, $P$ is locally closed.

Towards a contradiction therefore, let us assume that there exists a prime ideal $Q\supsetneq P$ with $Q\cap R=P\cap R=I$.
It follows that $F[x;\delta]$ is not simple: under the isomorphism $R[x;\delta]/J\cong S[x;\delta]$, $Q$ corresponds to a nonzero prime ideal $\tilde Q$ in $S[x;\delta]$ whose intersection with $S=R/I$ is trivial, so that $\tilde Q$ lifts to a nonzero prime ideal $Q_0$ in $F[x;\delta]$.
On the other hand, it is well-known that, as $F$ is a field of characteristic zero, if $\delta$ is nontrivial on $F$ then $F[x;\delta]$ is a simple ring (indeed this is a consequence of the fact that $F[x;\delta]$ is a left and right PID, see \cite[\S2.1]{SvdP}).
Thus, $\delta$ is trivial on $F$ and so $F[x;\delta]=F[x]$ is a PID.
So $P_0$, the lift of $\tilde P$ from $S[x;\delta]$ to $F[x;\delta]$, as it is properly contained in $Q_0$, must be $0$.
That is, $F[x;\delta]$ is a localisation of $R[x;\delta]/P$.
Hence, $\Frac\big(R[x;\delta]/P\big)=F(x)$, contradicting the rationality of $P$.

For the ``in particular" clause, note that $R[x;\delta]$ satisfies the Nullstellensatz -- by~\cite[Theorem~2]{Irving} for example -- and hence we already know that local-closedness implies primitivity and primitivity implies rationality.
\end{proof}

\begin{corollary}
\label{skewhopfdme}
Suppose $k$ is an %uncountable
algebraically closed field of characteristic zero and $R$ is a commutative affine Hopf $k$-algebra equipped with a $k$-linear derivation $\delta$ that is also an $a$-coderivation for some group-like $a\in R$.
Then $R[x;\delta]$ satisfies the $\dme$.
\end{corollary}

\begin{proof}
Theorem~\ref{twistdgroupDME-ring} together with Proposition~\ref{ddme-sdme}.
\end{proof}

A special case of Corollary~\ref{skewhopfdme} is when $R$ is a differential Hopf $k$-algebra -- this yields Theorem~B1 of the introduction.
But the $\dme$ for $R[x;\delta]$ in that case is easier: one uses only Theorem~\ref{dgroupDME-ring} and the material in Section~\ref{atwist} is not necessary.

\medskip
\subsection{The case when $\delta$ is inner}
If $\sigma$ is an automorphism of $R$, and $a\in R$, then the map $r\mapsto a(r-\sigma(r))$ is a $\sigma$-derivation on $R$.
Such $\sigma$-derivations are called {\em inner}.
Here is a sufficient criterion for a $\sigma$-derivation $\delta$ being inner.\footnote{For a more general statement in the noncommutative case, see~\cite[Lemma~2.4(b)]{Goodearl}.}

\begin{lemma}
\label{factinner}
Suppose $R$ is a commutative ring with an automorphism $\sigma$ and a $\sigma$-derivation $\delta$.
Suppose there exists an element $f\in R$ such that $f-\sigma(f)$ is a unit.
Then $\delta$ is inner.
\end{lemma}

\begin{proof}
It is easy to see, using the commutativity of $R$, that  $a:=\frac{\delta(f)}{f-\sigma(f)}$ witnesses the inner-ness of $(R,\sigma,\delta)$.
\end{proof}

When $\delta$ is inner the Dixmier-Moeglin equivalence for $R[x;\sigma,\delta]$ follows easily from known results.
It makes use however of one more notion:

\begin{definition} Let $A$ be a finitely generated algebra over a field $k$.  We say that a $k$-vector subspace $V$ of $A$ is a 
{\em frame} for $A$ if $V$ is finite-dimensional, contains $1_A$, and generates $A$ as a $k$-algebra.
\end{definition}

\begin{lemma}
Suppose $R$ is a commutative affine Hopf algebra over a field $k$ of characteristic zero and $\sigma$ is a $k$-algebra automorphism of $R$ satisfying statement~(2) of Corollary~\ref{ourbozz}.
Then there there is a frame for $R$ such that $\sigma(V)=V$.
\label{rem: frame}
\end{lemma}

\begin{proof}
Suppose $R=k[G]$ where $G$ is an affine algebraic group.
Then $G$ is linear and hence we may embed $G$ into ${\rm GL}_n$.
This gives us a frame $V$ of $R$ spanned by the restriction to $G$ of $1$, the coordinate functions $x_{i,j}$, and $\frac{1}{\det}$.
Now, $\Delta(x_{i,j})=\sum x_{i,k}\otimes x_{k,j}$ and $\Delta(\frac{1}{\det})=\frac{1}{\det}\otimes \frac{1}{\det}$.
So $\Delta(V)\subseteq V\otimes V$.
Statement~(2) of Corollary~\ref{ourbozz} then implies that $\sigma(V)\subseteq V$, and hence by finite-dimensionality $\sigma(V)=V$.\footnote{As a referee pointed out to us, the existence of a frame $V$ with $\Delta(V)\subseteq V\otimes V$ can be deduced for arbitrary finitely generated Hopf algebras by starting with any frame $W$ and extending it to a finite-dimensional subcoalgebra $V$ by the Finiteness Theorem for Coalgebras~\cite[Theorem~5.1.1]{montgomery}.}
\end{proof}

\begin{proposition}
\label{inner}
Suppose $k$ is an uncountable algebraically closed field of characteristic zero, $R$ is a finitely generated commutative $k$-algebra, $\sigma$ is a $k$-algebra automorphism of $R$ that preserves a frame, and $\delta$ is an inner $\sigma$-derivation on $R$.
Then $R[x;\sigma,\delta]$ satisfies the $\dme$.
\end{proposition}

\begin{proof}
When $\delta=0$ this is~\cite[Theorem~1.6]{BWW}; while the theorem there is stated for $k=\mathbb C$ it holds for any uncountable algebraically closed field.
But if $a\in R$ is such that $\delta(r)=a(r-\sigma(r))$ for all $r\in R$,
then $R[x;\sigma,\delta]=R[t;\sigma,0]$ where $t:=x-a$.
Indeed, 
\begin{eqnarray*}
tr &=&
(x-a)r\\
&=& \sigma(r) x + \delta(r) - ar\\
&=& \sigma(r)x-a\sigma(r)\\
&=& \sigma(r)t
\end{eqnarray*}
for all $r\in R$,
and $\{t^n:n\geq 0\}$ can be seen to be another left $R$-basis for $R[x;\sigma,\delta]$ using the fact that, for any polynomial $P$, $P(t)$ is equal to $P(x)$ plus terms of strictly lower degree.
So the inner case reduces to the case when $\delta=0$.
\end{proof}

\medskip
\subsection{The general case}
Let us fix from now on a %n uncountable algebraically closed 
field $k$ of characteristic zero.
Our proof of Theorem~\ref{hodme} will go via reducing either to the case when $\sigma=\id$ or when $\delta$ is inner.
But it will require some preparatory lemmas.
First, let us point out that statement~(2) of Corollary~\ref{ourbozz} forces $(R,\sigma)$ to be of a very restricted form.

\begin{lemma}
Let $G$ be a connected affine algebraic group over $k$ and $\tau:G\to G$ an automorphism of $G$ over $k$.
Let $R=k[G]$, and $\sigma=\tau^*$ the corresponding $k$-algebra automorphism of $R$.
If $(R,\sigma)$ satisfies statement~(2) of Corollary~\ref{ourbozz} then $\tau: G\to G$ is translation by some central element of $G(k)$.
\label{lem: translation}
\end{lemma}

\begin{proof}
Since $\chi: R\to k$ is a homomorphism there is some $c\in G(k)$ such that $\chi(f) = f(c)$ for all $f\in R$.
If we write $\Delta(f) = \sum f_1\otimes f_2$, then by the definition of the coproduct on $R$ we have
$f(ab) = \sum f_1(a)f_2(b)$ for all $a,b\in G$.
Now, property~(2) gives us that for all $a\in G$,
$\sigma(f)(a) = \sum \chi(f_1) f_2(a) = \sum f_1(c) f_2(a) = f(ca)$.
The other half of the equality in~(2) gives $\sigma(f)(a) = f(ac)$.
So $f(ca)=f(ac)$ for all $f\in R$, and hence $c$ is central in $G$.
On the other hand, $f(ca)=\sigma(f)(a)=f(\tau a)$ for all $f\in R$, so $\tau$ is translation by $c$.
\end{proof}

We will make use of the following notion.

\begin{definition}
Suppose $\sigma$ is an automorphism of a commutative ring $R$.
By a {\em $\sigma$-prime ideal} is meant a $\sigma$-ideal $I$ such that whenever $J$ and $K$ are $\sigma$-ideals with $JK\subseteq I$ then either $J\subseteq I$ or $K\subseteq I$.
\end{definition}

Note that a $\sigma$-prime ideal need not be prime.
But, at least in the case when $R$ is a commutative noetherian ring, a $\sigma$-prime ideal is radical; this follows from the fact that the nilpotent radical of $I$ is a $\sigma$-ideal and some power of it is contained in~$I$.
We will sometimes need to quotient out by $\sigma$-prime ideals that we do not know are prime, which means we will have to work with reduced difference rings that are not necessarily integral domains.
The following lemma about such difference rings will be very useful.

\begin{lemma} Suppose $R$ is a commutative ring endowed with an automorphism $\sigma$ such that $(0)$ is $\sigma$-prime.
If $0\neq f\in R$ satisfies $\sigma(f)\in Rf$ then $f$ is not a zero divisor in $R$.
\label{lem: eigenvector}
\end{lemma}

\begin{proof}
Let $J =Rf$. Then $J$ is a $\sigma$-ideal of $R$.  It follows that $K:=\{r\in R\colon rf=0\}$ is also a $\sigma$-ideal of $R$.  Then by construction $JK=(0)$. Since $(0)$ is $\sigma$-prime and $J$ is nonzero, we see that $K=(0)$ and so we obtain the desired result.
\end{proof}

Finally, we will make use of the following fundamental result on Ore extensions of commutative noetherian rings.

\begin{fact}[Goodearl~\cite{Goodearl}]
\label{fundamental}
Suppose $R$ is a commutative noetherian ring, $\sigma$ is an automorphism of $R$, and $\delta$ is a $\sigma$-derivation.
Suppose $P$ is a prime ideal of the Ore extension $R[x;\sigma,\delta]$, and let $I=P\cap R$.
Then one of the following three statements must hold:
\begin{itemize}
\item[I.]
$R[x;\sigma,\delta]/P$ is commutative.
\item[II.]
$I$ is a {\em $(\sigma,\delta)$-ideal} of $R$ -- that is, $I$ is preserved under $\sigma$ and $\delta$ --  and there is a prime ideal $I'$ of $R$ containing $I$ such that $\sigma(r)-r\in I'$ for all $r\in R$.
\item[III.]
$I$ is a $\sigma$-prime $(\sigma,\delta)$-ideal of $R$ and $IR[x;\sigma,\delta]$ is a prime ideal of $R[x;\sigma,\delta]$.
\end{itemize}
\end{fact}

We can now prove the theorem.

\begin{proof}[Proof of Theorem~\ref{hodme}]
We have that $R$ is a commutative affine Hopf $k$-algebra equipped with a $k$-algebra automorphism $\sigma$ and a $k$-linear $\sigma$-derivation $\delta$, and that the Ore extension $A:=R[x;\sigma,\delta]$ admits a Hopf algebra structure extending that of~$R$.
We wish to show that $A$ satisfies the $\dme$.
By~\cite[Theorem~2]{Irving} we have that $A$ satisfies the Nullstellensatz, and so it suffices to prove that if $P$ is a rational prime ideal of $A$ then $P$ is locally closed.

We first reduce to the case when $k$ is algebraically closed.
Since $R=k[G]$ where $G$ is an affine algebraic group, and hence smooth, $R$ is integrally closed.
Let $F$ denote the field of fractions of $R$ and let $F_0:=k^{\alg}\cap F$.
Since $R$ is integrally closed, $F_0\subseteq R$.
Since $F$ is a finitely generated extension of $k$, $F_0$ is a finite extension of $k$.  Let $R' : = R\otimes_{F_0} k^{\alg}$.
Since $F_0$ is relatively algebraically closed in $F$, we see that $R'$ is again an integral domain.
Thus $R'$ is a commutative affine Hopf $k^{\alg}$-algebra to which we extend $\sigma$ and $\delta$ by $k^{\alg}$-linearity.
Suppose we have proven the $\dme$ for $R'[x;\sigma,\delta]$.
Then Irving-Small reduction techniques (see Irving-Small \cite{IS} and also Rowen \cite[Theorem 8.4.27]{Rowen}) give that $A=R[x;\sigma,\delta]$ satisfies the $\dme$ over~$F_0$.
But since $F_0$ is a finite extension of $k$, we get the $\dme$ over $k$ also.

Next we reduce to the case when $k$ is uncountable (in order to be able to use Proposition~\ref{inner}).
Let $L$ be an uncountable algebraically closed extension of $k$.
Then since $k$ is algebraically closed we see that $R\otimes_k L$ is a commutative affine Hopf $L$-algebra to which we extend $\sigma$ and $\delta$ by $L$-linearity, and $B:=A\otimes_k L\cong (R\otimes_k L)[x;\sigma,\delta]$.
Assume the $\dme$ holds for $B$.
Let $P$ be a rational prime ideal of $A$ and let $Q=P\otimes_k L$.
Since $k$ is algebraically closed, $Q$ is a prime ideal of $B$.
Since $P$ is rational and $B/Q = (A/P)\otimes_k L$, we see that $Q$ is a rational.
Hence $Q$ is locally closed.
Since the primes in $A$ containing $P$ lift to primes in $B$ containing~$Q$, it follows that $P$ is locally closed in $A$.
So $A$ satisfies the $\dme$.

We may therefore assume that $k$ is uncountable and algebraically closed.

If $\sigma=\id$ then $\delta$ is a $k$-linear derivation on $R$ and statement~(1) of Corollary~\ref{ourbozz} tells us that it is also an $a$-coderivation on for some group-like $a\in R$.
It follows by Corollary~\ref{skewhopfdme} that $A=R[x;\delta]$ satisfies the $\dme$.
So we may assume $\sigma\neq\id$.

We may also assume that $A/P$ is not commutative.
Indeed, if it were, as $P$ is rational, we would have that $\Frac(A/P)\subseteq k$, so that $P$ is a maximal ideal and hence locally closed.

Write $R=k[G]$ where $G$ is a connected affine algebraic group over $k$.
By Lemma~\ref{lem: translation} we know that $\sigma=\tau^*$ where $\tau:G\to G$ is translation by a central (non-identity) element $c\in G(k)$.

Our next goal is to reduce to the case that $P\cap R=(0)$, though in order to obtain this we will have to give up on $R$ being an integral domain.
Let $I=R\cap P$.
We have already ruled out case~(I) of Fact~\ref{fundamental}.
On the other hand, case~(II) cannot hold: $\sigma$ would induce the identity map on $R/I'$, implying that $\tau$ is the identity on $V(I')$, which contradicts the fact that it is translation on $G$ by a non-identity element.
Hence case~(III) holds; $I$ is a $\sigma$-prime $(\sigma,\delta)$-ideal of $R$ and $J:=IA$ is a prime ideal of $A$.
Consider now the reduced quotient ring $\overline R:=R/I$ with the induced automorphism which we continue to denote by $\sigma$, and the induced $\sigma$-derivation which we continue to denote by $\delta$.
Let $\overline A=A/J\cong \overline R[x;\sigma,\delta]$ and $\overline P$ the image of $P$ in $\overline A$.
Since $J$ is contained in $P$, $\overline P$ is rational in $\overline A$ and it suffices to show that $\overline P$ is locally closed in $\overline A$.
Note that we have achieved $\overline P\cap \overline R=(0)$.

Next, we claim that there is some non zero-divisor $f\in\overline R$ such that $(\sigma,\delta)$ extends to $\widetilde R:=\overline R[1/f]$ and $\delta$ is inner on $\widetilde R$.
To see this, consider the frame $V$ for $R$, given by Lemma~\ref{rem: frame}, that is preserved by $\sigma$.
The image $\overline V$ of $V$ in $\overline R$ is then a frame for $\overline R$ that is also preserved by~$\sigma$.
Using $k=k^{\alg}$, let $f$ be an eigenvector for the action of $\sigma$ on $\overline V$, say $\sigma(f)=\lambda f$ for some $\lambda\in k^*$.
By Lemma~\ref{lem: eigenvector}, $f$ is not a zero divisor.
Moreover, the multiplicatively closed subset $\{1, f,f^2,\ldots\}$ of $R$ is preserved by $\sigma$, and hence by~\cite[Lemma 1.3]{Goodearl}, $(\sigma,\delta)$ extends uniquely to the localisation at this set, namely to $\widetilde R:=\overline R[1/f]$.
It remains to show that $f$ can be chosen so that $\delta$ is inner on $\widetilde R$.
If the eigenvalue $\lambda$ is not equal to $1$, then $f-\sigma(f)=(1-\lambda)f$ is a unit in $\widetilde R$, and we get $\delta$ inner by Lemma~\ref{factinner}.
So suppose that $1$ is the only eigenvalue for $\sigma$ on $\overline V$.
Note that $\sigma$ is not the identity operator on $\overline R$ because $\tau$ is not the identity on $V(I)$.
Since $\overline V$ generates $\overline R$ as a $k$-algebra, $\sigma$ is not the identity on $\overline V$ either.
Hence there must be some Jordan block that is of size greater than one, but with eigenvalue $1$.
So we can choose the eigenvector $f$ in such a way that there exists nonzero $g\in \overline V$ with $\sigma(g)=g+f$.
Hence $g-\sigma(g)$ is a unit in $\widetilde R=\overline R[1/f]$, and so by Lemma~\ref{factinner} again, $\delta$ is inner on $\widetilde R$.

To prove that $\overline P$ is locally closed let us consider the following partition of the set of prime ideals of $\overline A$ that properly extend $\overline P$.
\begin{eqnarray*}
S_1
&:=&
\{Q\supsetneq P: Q\text{ prime, and no power of $f$ is in }Q\}\\
S_2
&:=&
\{Q\supsetneq P: Q\text{ prime, not in $S_1$, and $Q\cap \overline R$ is a $\sigma$-prime $(\sigma,\delta)$-ideal}\}\\
S_3
&:=&
\{Q\supsetneq P: Q\text{ prime, and not in $S_1$ or $S_2$}\}
\end{eqnarray*}
It suffices to show that for each of $i=1,2,3$, $\bigcap S_i\neq \overline P$.

For $i=2$, note that as $\sigma$-prime implies radical, we have that $f\in Q$ for all $Q\in S_2$, but $f\notin\overline P$ as $\overline P\cap \overline R=(0)$.

For $i=3$, applying Fact~\ref{fundamental} to $Q\in S_3$, we have that either $\overline A/Q$ is commutative or there is in $\overline R=R/I$ a prime ideal $\overline I:=I'/I$ extending $Q\cap\overline R$, and such that $\sigma$ is the identity on $\overline R/\overline I=R/I'$.
The latter case is impossible using again that $\sigma=\tau^*$ and $\tau$ is translation on $G$ by a non-identity element.
So $\overline A/Q$ is commutative for all $Q\in S_3$.
As $\overline A/\overline P=A/P$ is not commutative there exist $a,b\in\overline A$ such that $g:=[a,b]\notin\overline P$.
But $g\in Q$ for all $Q\in S_3$.

It remains therefore to consider $S_1$.
Let $$\widetilde A=\widetilde R[x;\sigma,\delta]=R[\frac{1}{f}][x;\sigma,\delta].$$
As $\delta$ is inner on $\widetilde R$, Proposition~\ref{inner} tells us that $\widetilde A$ satisfies the Dixmier-Moeglin equivalence.
As $\overline P\cap \overline R=(0)$, we know that no power of $f$ is in~$\overline P$, and hence $\widetilde P:=\overline P\widetilde A$ is a prime ideal.
As $\widetilde A$ is a localisation of $\overline A$ we have that $\widetilde P$ is rational, and hence locally closed.
If $Q\in S_1$ then $Q\widetilde A$ is a prime ideal properly extending $\widetilde P$.
So there is $\alpha\in \widetilde A\setminus\widetilde P$ such that $\alpha\in Q\widetilde A$ for all $Q\in S_1$.
For some $n\geq 0$, $f^n\alpha\in \overline A$.
So $f^n\alpha\in Q\widetilde A\cap \overline A=Q$.
But $f^n\alpha\notin \overline P$.
So $\bigcap S_1\neq \overline P$, as desired.
\end{proof}

%\bibliographystyle{plain}
%\bibliography{diffDME}

\end{document}